\documentclass[12pt]{article}
\usepackage{amsmath,amsthm,amssymb}
\usepackage{xcolor}
\RequirePackage[colorlinks,citecolor=teal,urlcolor=blue]{hyperref}
\usepackage[utf8]{inputenc}
\usepackage[T2A]{fontenc}
\usepackage[main=english,ukrainian]{babel}
\usepackage{comment}
\usepackage{graphicx}
\usepackage{tikz}
\usepackage{cite}

\newcommand\harmonic{{\cal H}}
\newcommand\tree{{\cal T}}
\newcommand{\stkout}[1]{\ifmmode\text{\sout{\ensuremath{#1}}}\else\sout{#1}\fi}
\allowdisplaybreaks
\newcommand\ftz{\footnotesize}
\newcommand\frct [2]
{\ensuremath{\raisebox{.3ex}{\ftz #1}\!/\!\raisebox{-0.3ex}{\ftz#2}}}

\newcommand\normal{{\cal N}}

\newcommand\Poi{{\rm Poi}}
\newcommand\Ber{{\rm Ber}}

\newcommand\Dirichlet{{\rm Dirichlet}}
\newcommand\eqlaw{\buildrel {\cal L} \over = }

\newcommand\convD{{\buildrel {\cal L} \over \longrightarrow}}
\newcommand\given{\, \vert \, }
\newcommand\E{{\mathbb E}}
\newcommand\V{{\mathbb V{\rm ar}}}
\newcommand\prob{{\mathbb P}}
\newcommand\Cov{{\mathbb C{\rm ov}}}
\newcommand{\Prob}{{\mathbb{P}}}
\newcommand{\R}{\mathbb{R}}
\newcommand{\N}{{\mathbb N}}

\newcommand{\op}{\mathrm{op}}
\newcommand{\eins}{\mathrm{1}}
\newcommand{\Var}{\V}
\newcommand{\Dir}{\mathrm{Dir}}
\newcommand{\Beta}{\mathrm{Beta}}

\newtheorem{theorem}{Theorem}[section]
\newtheorem{corollary}{Corollary}[section]
\newtheorem{lemma}{Lemma}[section]
\newtheorem{proposition}{Proposition}[section]
\newtheorem{remark}{Remark}[section]

\newcommand\almostsure{\buildrel a.s. \over \longrightarrow}
\newcommand\inprob{\buildrel {\mathbb P} \over \longrightarrow}

\newcommand\field{{\mathbb F}}
\begin{document}
\begin{center}
{\huge \bf Distances in $m$-ary recursive trees}

\bigskip
{\large Kiran Bhutani\footnote{Department of Mathematics and Statistics, The Catholic University of America, Washington, D.C. 20064, U.S.A.; Email: \url{bhutani@cua.edu} } \qquad Ravi Kalpathy\footnote{Department of Mathematics and Statistics, The Catholic University of America, Washington, D.C. 20064, U.S.A.; Email: \url{kalpathy@cua.edu} }
 \qquad Florian Lesny\footnote{Institute for Mathematics, J.W.~Goethe University, Frankfurt a.M., Germany; Email: \url{lesny@math.uni-frankfurt.de} }
\qquad Hosam Mahmoud\footnote{Department of Statistics, The George Washington University, Washington, D.C. 20052, U.S.A.; Email: \url{hosam@gwu.edu}}
 \qquad Ralph Neininger\footnote{Institute for Mathematics, J.W.~Goethe University, Frankfurt a.M., Germany; Email: \url{neininger@math.uni-frankfurt.de} }}

\medskip
\today
\end{center}
\bigskip\noindent
\section*{Abstract}
We introduce the $m$-ary recursive tree, a tree structure that starts with a single node. At each step of the growth process, a random node, called the recruiter, is selected, and $m$ new nodes are attached to this recruiter.

The main objective is an asymptotic analysis of distances in this structure. We first derive the distribution of the depth of nodes in such a tree as a convolution of independent (though not identically distributed) Bernoulli random variables. This representation yields normal and Poisson approximations with quantified errors 
for the depth. 
We also investigate the height of these trees, that is, the longest root-to-leaf distance, via a coupling to a weighted random recursive tree with suitably 
chosen weights. The coupling transfers the results on a strong law and tightness from the weighted
recursive tree to the $m$-ary recursive tree.

Furthermore, we establish a bivariate limit law for the Wiener index, that is, the sum of all pairwise distances, and the internal path length, that is, the sum of the depths of all nodes, using the contraction method. As a novel technical feature of our approach, we employ weighted norms instead of Euclidean norms within the contraction method.

\bigskip

\noindent{\bf AMS subject classifications:} 
Primary:  05C82, 
               05C12, 
               90B15; 
    Secondary: 60C05, 
 60F05. 

\medskip\noindent
{\bf Keywords:} Tree, network, 
distance in graph, Wiener index, central limit theorem,
Poisson approximation, contraction method, fixed point. 
\section{Introduction} 
The $m$-ary recursive tree grows according to an algorithm that adds $m$  nodes at each step.
The tree starts out with a single root node labeled 1. At time $n\ge 1$, an existing
node in the tree is selected uniformly at random as a {\em parent} for $m$ incoming nodes. 
The~$m$ nodes are attached as children
to the chosen parent and all receive the label $n+1$; the process is then repeated. One can think 
of choosing a parent for children as {\em recruiting}, and the selected parent is therefore a {\em recruiter}.

The case $m=1$ is the well-studied {\em standard} recursive tree~\cite{Drmota,Frieze,Hofri,Smythe}. 
Figure~\ref{Fig:rec} illustrates two steps of growth of a ternary ($m=3$) recursive tree.
While the first two trees are deterministic (occur with probability 1), the third tree 
(the rightmost in the figure) occurs with probability $\frct 1 4$. 

\bigskip
\begin{figure}[thb]
\begin{center}
\begin{tikzpicture}[scale=0.6]


\draw [ultra thick,fill=white] (13,8) circle [radius=0.45];
\node at (13, 8.9) {$\star$};
\node at (13, 8) {$1$};

\node at (19.5, 6.9) {$\star$};
\coordinate (A) at (18,8);
\coordinate (B) at (16.5,6);
\coordinate (C) at (18,6);
\coordinate (D) at (19.5,6);

\draw [ultra thick] (A)--(B);
\draw [ultra thick] (A)--(C);
\draw [ultra thick] (A)--(D);

\draw [ultra thick,fill=white] (18,8) circle [radius=0.45];
\draw [ultra thick,fill=white] (16.5,6) circle [radius=0.45];
\draw [ultra thick,fill=white] (18,6) circle [radius=0.45];
\draw [ultra thick,fill=white] (19.5,6) circle [radius=0.45];

\node at (18, 8) {$1$};
\node at (16.5, 6) {$2$};
\node  at (18, 6) {$2$};
\node at (19.5, 6) {$2$};


\coordinate (A) at (23,8);
\coordinate (B) at (21.5,6);
\coordinate (C) at (23,6);
\coordinate (D) at (24.5,6);
\coordinate (E) at (23.3,4);
\coordinate (F) at (24.5,4);
\coordinate (G) at (25.7,4);
\coordinate (H) at (22.1,4);
\coordinate (I) at (26.9,4);

\draw [ultra thick,fill=white] (18,8) circle [radius=0.45];
\draw [ultra thick,fill=white] (16.5,6) circle [radius=0.45];
\draw [ultra thick,fill=white] (18,6) circle [radius=0.45];
\draw [ultra thick,fill=white] (19.5,6) circle [radius=0.45];

\node at (18, 8) {$1$};
\node at (16.5, 6) {$2$};
\node  at (18, 6) {$2$};
\node at (19.5, 6) {$2$};

\draw [ultra thick] (A)--(B);
\draw [ultra thick] (A)--(C);
\draw [ultra thick] (A)--(D);
\draw [ultra thick] (D)--(E);
\draw [ultra thick] (D)--(F);
\draw [ultra thick] (D)--(G);

\draw [ultra thick,fill=white] (23,8) circle [radius=0.45];
\draw [ultra thick,fill=white] (21.5,6) circle [radius=0.45];
\draw [ultra thick,fill=white] (23,6) circle [radius=0.45];
\draw [ultra thick,fill=white] (24.5,6) circle [radius=0.45];

\draw [ultra thick,fill=white] (23.3,4) circle [radius=0.45];

\draw [ultra thick,fill=white] (24.5,4) circle [radius=0.45];
\draw [ultra thick,fill=white] (25.7,4) circle [radius=0.45];

\node at (23, 8) {$1$};
\node at (21.5, 6) {$2$};
\node  at (23, 6) {$2$};
\node at (24.5, 6) {$2$};

\node at (23.3, 4) {$3$};
\node  at (24.5, 4) {$3$};
\node at (25.7, 4) {$3$};

\end{tikzpicture}
\end{center}
  \caption{Two steps of growth of a ternary recursive tree. The recruiting nodes are flagged 
  with asterisks.}
  \label{Fig:rec}
\end{figure}
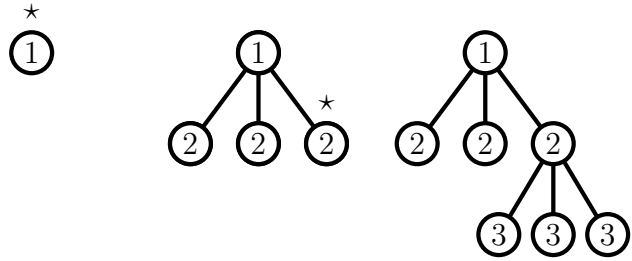

Recursive trees offer a powerful tool in various modern applications. 
Allowing multiple additions at each step widens the scope of application of recursive trees. For example, in the context of chain letters~\cite{Gastwirth}, the $m$-ary option allows a letter holder to sell or distribute 
multiple copies of the letter simultaneously to a single or multiple new owners. 
The structure of the $m$-ary recursive tree may suggest an accelerated Union-Find algorithm, a point we shall return to in the future. 
\section{Scope of the research}
The depth of a node in a rooted tree is the length of the path (measured in edges) joining the node to the root.  The depth of a node in a standard recursive tree ($m=1$) and in
a few other variants is shown to be distributed like a convolution of independent Bernoulli random 
variables~\cite{Colin3,Dobrow,Lyon1,Lyon2,Nakata}. 

In this paper, we show that this useful property carries over to $m$-ary recursive trees and use it advantageously to find  an exact distribution, a central limit theorem and Poisson approximations (both with rates of convergence) for the depth.

We study the height via a coupling to a weighted random recursive tree~\cite{Pain1,Pain2} with suitably chosen weights. The coupling transfers the results on a strong law and tightness from the weighted
recursive tree to the $m$-ary recursive tree.

The sum of the depths of all the nodes, called the internal path length,
and the inter-distances between all pairs of nodes, called the Wiener index,
are two significant topological indices. 
We find their exact mean values; our main contribution is a characterization of their joint bivariate 
limit distribution as the fixed-point solution to a distributional equation. A novel aspect of
the approach is the use of weighted norms in the contraction method instead of the more standard Euclidean norms. 

The papers~\cite{Curien,Mailler,Ward} and the references therein analyze models that are different but quite related to our work.
The families of graphs considered by these authors grow by gluing an additional cluster on the tree at each epoch, instead of recruiting a vertex 
(as in the standard recursive tree) or a bunch 
of disconnected vertices (as we do here). 
\begin{remark}
According to the construction algorithm,  in an m-ary recursive tree,
the labels on a root-to-leaf path form an increasing sequence of numbers.
Some authors call such trees ``increasing,'' see~\cite{Bergeron}. The
$m$-ary recursive tree  we propose adds another species to the class of increasing trees.  
\end{remark}
For integer $r\ge 1$ and real $y>0$, in the sequel  
we use the notation 
$$\harmonic_n^{(r)} (y) = \sum_{i=1}^n \frac 1 {(i+y)^r}$$
for the $n$th generalized harmonic number of order $r$.
It is well known that, as $n\to\infty$, we have the asymptotic approximations 
$${\cal H}_n^{(1)}(y) \sim \ln (n), \qquad {\cal H}_n^{(2)}(y) \sim \Psi (1, 1+y),$$
where $\Psi(1, 1+y)$ is the trigamma function~\cite[pp.~260--263]{Stegun}.

In what follows we use\, $\eqlaw$\, for exact equality in law, and $\convD$ for convergence 
in law (distribution). The symbols $\inprob$ and $\almostsure$ respectively stand for convergence in probability
and convergence almost surely.
\section{Distances in an $m$--ary recursive tree}
We are concerned with several types of distances. We define and discuss them in this section and give an overview of their significance. 

The {\em depth}, $D_n$, of a node labeled with $n$ in an $m$-ary recursive 
tree is its edge distance to the root. 
For instance, in the trees in Figure~\ref{Fig:rec} the root is at depth 0,  all 
the nodes labeled with 2 are at depth 1 and all the nodes labeled with 3 are at depth 2.

The depth of a randomly chosen node has been addressed in a more general model by \cite{Mailler} (see Remark \ref{rem_mailler}). This is a notion of depth different from $D_n$.
The following representation of $D_n$ as a Bernoulli convolution
is a generalization of theorems in~\cite{Dobrow,Nakata}. A more general construction is discussed in~\cite{Colin3}, leading to this Bernoulli convolution. The appropriate parameterization 
of Desmarais's hooking network~\cite{Colin3} yields  
a representation of the convolution from which we extract the exact 
distribution of $D_n$ in an easy-to-compute form.

While the depth is a local property of a node, several other distances are defined by the overall structure of the tree. As such, they are measures of global performance. We focus on the internal path length, the height and the Wiener index.

The {\em internal path length}, $P_n$, is the sum of all node distances from the root. The three trees in  Figure~\ref{Fig:rec}, from left to right, have internal path lengths 0, 3, and 9, respectively.

The {\em Wiener index}, $W_n$, is the sum of all inter-node distances. The three trees in  Figure~\ref{Fig:rec}, from left to right, have Wiener indices 0, 9, and 42, respectively.

The {\em height}, $H_n$, is the longest root-to-leaf distance in an $m$-ary recursive tree,
when the highest label in it is $n$. 
For instance, the rightmost tree in~Figure~\ref{Fig:rec} is of height 2.
The height of the standard recursive tree represents the 
worst-case behavior for some algorithms,
such as Union-Find~\cite{Devroye}.

Many of the random variables in this manuscript depend on~$m$. 
For instance, the depth of the nodes labeled $n$
may appropriately be called~$D_{m,n}$. Note that all nodes labeled $n$ are at the same depth (same distance from the root). For lighter notation, we leave $m$ out of all these variables, as $m$ is held fixed anyway. The presence of $m$ in the notation should be implicitly understood.

Let $\tau_n$ be the number of nodes in the tree,  when the highest label in it is~$n$. We have
$$\tau_n = \tau_{n-1} + m = (n-1)m + 1.$$ 
$$\tau_n = \tau_{n-1} + m = (n-1)m + 1.$$ 
\section{The depth as a Bernoulli convolution}
In this section, we repeatedly exploit that
the $m$--ary recursive tree may be viewed as a hooking network in the sense of~\cite{Bhutani1,Bhutani2,Bhutani3,Colin1,Colin2,Colin3}, wherein a large network is built
by hooking small components. The $m$--ary recursive tree is built by repeatedly hooking
copies of the complete bipartite graph $K_{1,m}$. In the language of these sources, the lone vertex in one part of $K_{1,m}$ acts as the hook
and the recruiter acts as a latch.

We denote a Bernoulli random variable with parameter $p$ by Ber($p$).

\begin{theorem}
\label{Thm:Bernoulli}
Let $D_n$ be the depth of a node labeled with $n\ge 2$ in an $m$-ary recursive 
tree.\footnote{\label{foot_6}Interpret an empty sum
as 0. Specifically,
when $n=2$, we take the sum of the Bernoulli random variables to be 0.} 
We then have
$$D_n  \eqlaw 1 + \Ber \Big( \frac m {m+1}\Big) +  \Ber \Big( \frac m {2m+1}\Big) +
\cdots + \Ber \Big( \frac m {m(n-2)+1}\Big),$$
for independent Bernoulli random variables in the convolution.
\end{theorem}
\begin{proof}
The Bernoulli convolution in the statement is a special case of the convolution underlying the hooking networks in \cite[Lemma 1]{Colin3}. The parameters in  \cite{Colin3} to cover the  random $m$-ary recursive tree are $\chi =0$, $\rho=1$, $W_0=1$, $W_n=m$ for $n\ge 1$, $\Delta'_{n,k}=1$, and $J_{n,k}$, which are independent and distributed as $\mathrm{Ber}(W_k/(mk+1))$. (Note that the $(n-1)$st insertion corresponds to a node labeled $n$ in our setting.)
\end{proof}
Recall that $\mbox{Ber} (p)$ has the moment generating function
$ 1-p + pe^t$. Hence, the Bernoulli convolution in Theorem \ref{Thm:Bernoulli} implies that the moment generating function $\phi_n(t) = \E[e^{D_nt}]$ has the form
\begin{equation}
\phi_n(t)  = e^t\prod_{k=2}^{n-1}\Bigl(1 - \frac m {\tau_k}
       + \frac m {\tau_k}\, e^t\Bigr), \qquad \mbox{for\ } n \ge 3;
       \label{Eq:mgf}
\end{equation}  
The result extends to cover $n=2$ under interpretation that an empty product is equal to 1.

\begin{remark}\label{rem_mailler}
In  \cite{Mailler}, a more general model called a weighted random recursive tree (WRRT) is studied. Among the many results obtained, the depth of nodes is addressed. 
Given a positive sequence of weights $(w_i)_{i \ge 1}$, the tree starts with a root node of weight $w_1$. At each subsequent step, a new node is attached to one of the existing nodes, chosen at random with probability proportional to their weights and obtains the corresponding weight of $(w_i)_{i \ge 1}$. More details can be found in \cite{Mailler}. 

In \cite{Mailler}, the depth of a node  chosen  randomly with probabilities proportional to the weights in the WRRT is studied. 
The choice $w_1=1$ and $w_i=m$, for $i \ge 2$, 
relates to the present model of the random $m$-ary recursive tree and yields the depth of a uniformly chosen node in the random $m$-ary recursive tree. 
In particular, \cite[Corollary 8]{Mailler} specialized to the case of a uniformly chosen node in the random $m$-ary recursive trees shows that its depth has the representation 
    \begin{align*}
\Ber \Big( \frac m {m+1}\Big) +  \Ber \Big( \frac m {2m+1}\Big) +
\cdots + \Ber \Big( \frac m {(n-1)m+1}\Big)
    \end{align*}
    for independent Bernoulli random variables in the convolution.
\end{remark} 
\begin{corollary}
\label{Cor:avevar}
For $n\ge 2$, we have
\begin{align*}
\E[D_n] &= 1 + \harmonic_{n-2}^{(1)}\Big(\frac 1 m\Big) \sim \ln (n), \qquad \mbox{as \ }n\to\infty;\\
\V[D_n] &=  \harmonic_{n-2}^{(1)}\Big(\frac 1 m\Big) -  \harmonic_{n-2}^{(2)}\Big(\frac 1 m\Big)
      \sim \ln ( n ), \qquad  \mbox{as \ }n\to\infty.
\end{align*}
\end{corollary}
\begin{corollary}
\label{Cor:weaklaw}
$$\frac {D_n}{\ln (n)} \inprob 1, \qquad  \mbox{as \ }n\to\infty.$$
\end{corollary}
The Bernoulli convolution in Theorem~\ref{Thm:Bernoulli} is not in closed form and may
be  computationally intensive. We can extract a formula from the representation (\ref{Eq:mgf}) in terms of standard functions. The formula involves the well-tabulated gamma functions and
generalized Stirling numbers that appear as coefficients in the expansion of the polynomial 
$u(u + a)(u+ 1 + a) \cdots (u + r-2 + a)$ captured in the notation
$$u(u + a)(u+ 1 + a) \cdots (u + r-2 + a) = \sum_{j=1}^r \Big[{r\atop j}\Big]_a u^j;$$
see~\cite{Janjic}.
\begin{corollary}
\label{Cor:exactdist}
For $n\ge 2$ and $1\le k\le n-1$, the exact distribution of the depth is 
$$\Prob(D_n =k) = \Big[{n-1 \atop k} \Big]_{1/m} \, \frac {\Gamma \big(1 + \frac 1 m\big)}{\Gamma\big(n-1 + \frac 1 m\big)}.$$
\end{corollary}
\begin{proof}
Put $t = \ln (u)$ in~(\ref{Eq:mgf}), to get the probability generating function
\begin{align*}
\E[u^{D_n}]  &= u \prod_{k=2}^{n-1}\Bigl(1 - \frac m {\tau_k}
       + \frac m {\tau_k}\, u\Bigr) =   u\, \frac{\prod_{k=2}^{n-1} \big(u+ k-2 +\frac 1 m\big)}
                     {\prod_{k=2}^{n-1} \big(k-1+\frac 1 m\big)}.
       \end{align*}
Complete the product in the denominator into a gamma function and extract 
$\Prob(D_n=k)$, 
the coefficient of 
 $u^k$ in
 $$u\Big(u + \frac 1 m\Big) \Big(u+ 1 + \frac 1 m\Big) \cdots \Big(u + n-3 + \frac 1 m\Big) 
 =  \sum_{j=1}^{n-1} \Big[{n-1\atop j}\Big]_{1/m}u^j.$$      
\end{proof}
\begin{remark}
At $m=1$, the notation $[{n-1 \atop k}]_1$ is reduced to an ordinary signless Stirling number
of the first kind and the exact probability distribution in Corollary~\ref{Cor:exactdist} recovers
the result in~\cite{Mahmoud1991}.   
\end{remark}
\begin{remark}
The distribution in  Corollary~\ref{Cor:exactdist} may not require calculation. It can 
be evaluated by matching the elements on the right-hand side with 
suitable sources. For instance, for $m=2,$ the generalized Stirling number 
$\Big[{n-1\atop k}\Big]_{1/2}$
is $A(n-2,k-1) / 2^{n-k-1}$, where the $A$  is Sequence 
A028338 in the On-Line Encyclopedia of Integer Sequences (https://oeis.org), and we can look up
the values of the gamma functions in~\cite[Table 6.3]{Stegun}. 
 
 As an example, we have
$$\Prob(D_{40}=6)
        =\frac {A(38,5)} {2^{33}} \times  \frac{\Gamma(3/2)}{\Gamma(79/2)}\approx \boxed{0.1524552895};$$
the OEIS triangular array gives
$$A(38,5) =4811142885595871475712547250061738279279070812198762500.$$
\end{remark}
In the sequel, we need an expansion of the expectation $\mathbb{E}[P_n]$, 
which we obtain from the expected depths of the nodes. 
\begin{corollary} 
\label{Cor:IPL}
    Let $P_n = m\sum_{k=1}^n D_k$ be the internal path length when the highest
    index in the tree is $n$. We have 
\begin{align*} 
    \gamma_n:=\mathbb{E}[P_n]&=(m(n-1)+1)\mathcal{H}_{n-1}^{(1)}\Big(\frac{1}{m}\Big)\\
    &= mn\ln(n)-m\psi\Big(\frac{m+1}{m}\Big)n+o(n), 
    \end{align*} 
    where $\psi(x):=\frac{d}{dx}(\ln\Gamma(x))$ is the digamma function. 
\end{corollary} 
\begin{proof}
  Using the representation of the depth from Theorem \ref{Thm:Bernoulli}, we obtain
    \begin{align*} 
        \mathbb{E}[P_n]&=m(n-1)+m\sum_{i=3}^{n} (n+1-i)
                    \E\Big[\Ber\Big(\frac{m}{m(i-2)+1}\Big)\Big]\\
        &=m(n-1)+m\sum_{i=3}^{n} \frac{m(n+1-i)}{m(i-2)+1}. 
    \end{align*} 
    Computing the second term in the latter display yields
    \begin{align*} 
        m\sum_{i=3}^{n} \frac{m(n+1-i)}{m(i-2)+1}= \big(m(n-1)+1\big)\mathcal{H}_{n-1}^{(1)}
        \Big(\frac{1}{m}\Big)-m(n-1). 
    \end{align*} 
    The asymptotics of $\mathcal{H}_{n}^{(1)}(y)$ are given by 
    $\ln(n)-\psi(1+y)+O(\frac{1}{n})$, which implies
    \begin{align*} 
        \mathbb{E}[P_n]=mn\ln(n)-m\psi\Big(\frac{m+1}{m}\Big)n+o(n).
    \end{align*} 
\end{proof}
\begin{remark}
Bernoulli convolutions of the depth are reported in~\cite{Dobrow,Mailler,Lyon1,Lyon2,Nakata} in a variety of recursive tree structures that are different from the $m$-ary model considered here. Some 
alternative models assume general weights~\cite{Mailler,Pain1,Pain2}, also called affinities~\cite{Nakata}.
\end{remark}
\subsection{A Gaussian law for the depth}
In this subsection, we discuss the normality of $D_n$. 
Referring to hooking networks, we adapt the results in~\cite{Colin1,Colin2,Colin3} to get a central limit theorem for the depth.

\begin{theorem}
\label{Thm:CLT}
For an $m$-ary recursive tree with highest label $n$, 
we have the convergence
$$ \frac{D_n - \ln (n)} 
      {\sqrt {\ln (n)}} \ \convD \ \normal (0,1), \qquad \mbox{as\ } n \to\infty.$$    
\end{theorem}
\begin{proof}
This is a special case of Theorem~5.4 in~\cite{Colin1}, with 
only one block in the collection of hooks with $\Delta = 1$.
\end{proof}
We add a rate of convergence in the following.
\begin{proposition}
 The normalized depth $Z_n :=\frac{D_n - \mu_n}{\sigma_n}$ satisfies
$$\sup_{x\in \mathbb R}\big|\prob (Z_n \le x) - \Phi(x) \big| 
                        = \Theta \Big(\frac 1 {\sqrt{\ln (n)}} \Big).$$
\end{proposition}
\begin{proof}
Let $X_k$ be the $k$th Bernoulli random variable in the convolution. 
For the rate of convergence, we use
Berry-Esseen's bound~\cite{Berry,Esseen}
(see also~\cite[Theorem 7.6.2, p.~356]{Gut2}),
which, for some positive constant, $C$, gives the bound
\begin{align*}
\sup_{x\in \mathbb R}\big|\prob (Z_n \le x) - \Phi(x) \big|
&\le C \,
\frac {\sum_{k=2}^{n-1} \E\big|X_k - \E[X_k]\big|^3} 
       {\big(\sum_{k=2}^{n-1}\V[X_k]\big)^{3/2}} \\
&\le C\,
\frac {\sum_{k=2}^{n-1} 
      \big( \frac m {\tau_k}\big)^3\big(1-\frac m {\tau_k})  
         + \sum_{k=2}^{n-1}\big(1-\frac  m {\tau_k})^3 \, \frac m {\tau_k}} 
    {\big(\sum_{k=2}^{n-1} \frac m {\tau_k}
              \big(1-\frac m {\tau_k}\big)\big)^{3/2}} \\
&= \frac {\Theta \big(\ln (n)\big)} {\big(\Theta \big(\ln (n)\big)\big)^{3/2}} \\
&=\frac 1 {\Theta (\sqrt{\ln (n)}\,)}.
\end{align*} 
\end{proof}

\subsection{Poisson approximation}
For two probability measures $\mathbb Q_1$ and $\mathbb Q_2$ (both over $\mathbb N \cup \{0\}$),  the total variation distance
is defined as
$$
d_{TV}(\mathbb Q_1,\mathbb Q_2)=  \sup_{A \subseteq \mathbb N \cup \{0\}} \big|\mathbb Q_1(A) 
                     - \mathbb Q_2(A)\big|= \frac 1 2 
   \sum_{j=0}^\infty \big|\prob(X=j) - \prob(Y=j) \big|,$$
for random variables $X$ and $Y$, with the respective distributions $\mathbb Q_1$
and $\mathbb Q_2$.
A main tool to bound  the total variation distance is the following 
theorem~\cite{Hall}, see also~\cite[Theorem 2.M, p.~34]{Holst}.
\begin{theorem} 
\label{Thm:Holst}
Let $X_1, \ldots, X_n$ be independent
Bernoulli random variables, with $X_k\ \eqlaw \ \Ber(p_k)$,
for $k=1, \ldots, n$, 
and let $S_n = \sum_{k=1}^n X_k$
be the sum of these variables. Define
$$\lambda_{n,1} = \sum_{k=1}^n \E[X_k] = \sum_{k=1}^n p_k,
\qquad
and\qquad  \lambda_{n,2} =\sum_{k=1}^n p_k^2.$$
Then, we have
$$d_{TV}\big({\cal L}(S_n),  {\cal L}\big(\Poi ({\lambda_{n,1}})\big)
         \big) \le (1 - e^{-\lambda_{n,1}}) \, \frac           
    {\lambda_{n,2}} {\lambda_{n,1}}, $$
where ${\cal L}(\cdot)$ denotes the probability law.
\end{theorem}

For a Poisson approximation, from Corollary~\ref{Cor:avevar}, we compute
\begin{align*}
\lambda_{n,1} &= \E[D_n] =                 
                1 + \harmonic_{n-2}^{(1)} \Big(\frac 1 m\Big), \\
\lambda_{n,2} &= 1 + \sum _{k=2}^{n-1} \Big(\frac m {\tau_k}\Big)^2
                        = 1 + \harmonic_{n-2}^{(2)} \Big(\frac 1 m\Big).
\end{align*}                        
We get
$$d_{TV}\Big({\cal L} (D_n), {\cal L} \Big(\Poi
                     \Big(1 + \harmonic_{n-2}^{(1)} \Big(\frac 1 m\Big)\Big)\Big) \Big)
        \le \frac {1+\psi(1, 1 + \frac 1 m)} {1+\harmonic_{n-2}^{(1)} (\frac1 m)}
                   = \mathrm{O}\Big(\frac{1}{\ln (n)}\Big),$$
      as  $n \to\infty$.

It is a somewhat slow rate to get good approximations for small $n$. 
For example, when~$n$ equals one million, the distribution function
of $D_{1000000}$ in a ternary recursive tree is uniformly not different from that of 
$\Poi(14.94754316)$
by as much as about $0.1401967603$.

\section{The height}
\label{Sec:height_1}

We reduce the analysis of the height $H_n$ of the $m$-ary recursive tree to the result on the height of a WRRT tree with suitably chosen weights. 

In the $m$-ary recursive tree, at each step $k\ge 2$, the $m$ nodes carrying label $k$ have the same parent and therefore the same depth. We may thus contract each such batch into a single 
{\em super node}. 
This preserves the depth of each batch and hence the height.
The resulting tree has vertex set $\{1,\dots,n\}$ and evolves as a weighted recursive 
tree with the deterministic weights
\begin{align}\label{our_weight}
w_1=1,\qquad w_k=m,\quad k\ge 2,
\end{align}
that is, at the various steps, recruiting
nodes are chosen with probabilities proportional to their weights; see~\cite{Pain2} for general weighted recursive trees. Various results on the height of weighted recursive trees from \cite{Sen21,Pain1,Pain2} therefore apply to the $m$-ary recursive tree.

\begin{corollary}\label{cor:height-asymptotics}
For every fixed $m\ge 1$,
\[
\frac{H_n}{\log n}\to e
\qquad\text{a.s.}
\]
as $n\to\infty$.
Furthermore, the sequence
\[
\Big(H_n - e\log n+\frac32\log\log n\Big)_{n\ge 2}
\]
is tight.
\end{corollary}

\begin{proof}
Theorem 5.1 in \cite{Sen21} applies with the choices $\gamma=1$ and $z_+=1$ and implies the stated almost sure convergence. For the tightness result, Theorem 1.1 in \cite{Pain1} applies with $\gamma=1$ and $\theta=1$, since assumptions $(\mathcal{H}_{1,\gamma})$ 
and $(\mathcal{H}_{2})$ are satisfied for the sequence of weights $(w_k)_{k\ge 1}$ in (\ref{our_weight}).
\end{proof}

Further results, for example upper bounds on the right tail of the height, can be found in \cite{Pain2}.
\section{The average of the Wiener index}
\label{Sec:Wiener}
Let~$\tree_n$ be the $m$-ary recursive tree when the highest index in it is $n$, 
and let~$V_n$ be its set of vertices. 
We develop the expectation of $W_n$. Recall that the number of nodes in $\tree_n$ is $\tau_n=m(n-1)+1$. 
\begin{proposition}
\label{Prop:Wiener}
Let $W_n$ be the Wiener index of a random $m$-ary recursive tree with highest index $n$.
We have
\begin{align*}
\E[W_n] &= \big(m(n -1) +1\big) (n m +1) \Big(\harmonic_n^{(1)}\Big(\frac 1 m\Big) - \frac {mn}{mn+1}\Big) \\
&\sim m^2 n^2\ln (n),\qquad as\ n \to \infty.
\end{align*} 
\end{proposition}  
\begin{proof} 
Let $v$ be a node in the tree $\tree_{n-1}$, and $d(v,x)$ be the
distance between the nodes $v$ and $x$
in the tree.
When the batch of $m$ nodes to be labeled~$n$ is recruited by $L_{n-1}\in V_{n-1}$, any  
 node $v$  in $\tree_{n-1}$ is at distance $1+ d(L_{n-1}, v)$ from
any node in the new batch.
Any pair of nodes in the incoming batch contributes an additional $2$ inter-distance
to the Wiener index of $\tree_n$, and there are ${m\choose 2}$ such contributions.

According to this counting argument,  
the Wiener index satisfies the conditional stochastic recurrence
$$W_n\given \field_{n-1}, L_{n-1} 
                 =W_{n-1}  +  2{m \choose 2} 
                           + m \sum_{v \in V_{n-1}}
                           \big(1 + d(L_{n-1} ,v)\big).$$
Take expectations with respect to $L_{n-1}$ to get
\begin{align*}
W_n \given \field_{n-1} &= W_{n-1}  +  2{m \choose 2} 
                           + m\Big(\tau_{n-1} + \frac {1} {\tau_{n-1}}\sum_{v \in V_{n-1} } 
                                 \sum_{\ell\in V_{n-1}}d(\ell,v)\Big)\nonumber \\  
          &= W_{n-1}  + m(m-1)
                           + m \Big(m(n-2)+1 + \frac {2\, W_{n-1}} {m(n-2)+1}\Big).                                       
\end{align*}     
A second expectation gives the unconditional expectation,    
which we can organize as a recurrence for the first moment:
$$\E[W_n] = \Big(\frac {m n + 1} {m(n-2)+1}\Big)\, \E[W_{n-1}]  + m^2 (n-1).$$
This is a standard linear recurrence of the form
$$\E[W_n] = g_n\, \E[W_{n-1}]  +h_n,$$
with 
$$g_n = \frac{m n + 1} {m(n-2)+1}, \qquad h_n = m^2 (n-1).$$
Iterating the recurrence, we obtain\footnote{Interpret an empty product 
as 1. Specifically,
when the summation index~$i$ is $n$, the product index $j$ goes from $n+1$ to $n$; we take this product to be 1.} 
$$\E[W_n] = \sum_{i=2}^n h_i \Big(\prod_{j=i+1}^n g_j\Big)  
       +  \Big(\prod_{j=1}^n g_j\Big)  \E[W_1]. $$
With $\E[W_1] = 0$, and the specific $g_j$ and $h_i$ in the recurrence, we get
\begin{align*}
\E[W_n] &= \sum_{i=2}^n  \frac {\big(m(n-1)+1\big)(mn+1)} 
          {\big(m(i-1)+1\big)(mi+1)} \, m^2(i-1)\\
             &=  \big(m(n-1)+1\big)(mn+1) \sum_{i=2}^n  \frac {i-1} {(i-1+1/m)(i+1/m)}\\
             &=  \big(m(n-1)+1\big)(mn+1) \sum_{i=2}^n \Big( \frac {(m+1)/m} {(i+1/m)}
                       -  \frac {1/m} {(i-1+1/m)}\Big).
\end{align*}       
We now discern the presence of generalized harmonic numbers; the result follows after some
algebraic simplification.          
\end{proof}
\begin{remark}
The generalized harmonic number $\harmonic_n^{(1)}(0)$ is the ordinary harmonic number of order $n$.
In the special case of the standard recursive tree (m=1), Proposition~\ref{Prop:Wiener} recovers
Neininger's result~\cite{Neininger}:
$$\E[W_n] = n^2 \harmonic_n^{(1)}(0)-2 n^{2} +n \harmonic_n^{(1)}(0) \sim n^2\ln (n).$$
\end{remark}
\begin{remark}
In~\cite{Achuna}, the authors consider hooking small graphs that are copies of a starting seed and establish a formula for the average  Wiener index.

When the seed graph is the complete graph on two vertices, such a hooking graph
becomes a standard recursive tree ($m=1$). 
The average Wiener index result there coincides with Proposition~\ref{Prop:Wiener}, too. 
\end{remark}
\section{Interplay between Wiener index and the internal path length: the distributional view}
For the derivation of the expectation $\E[W_n]$ in Section~\ref{Sec:Wiener}, the transition from~$n$ to $n-1$ is considered (sometimes called the forward view). 

Toward a limit law via contraction, we exploit the recursive nature of the random tree and decompose it at the root into $m+1$ trees (sometimes called the backward view), 
each being isomorphic to an  $m$-ary recursive tree and conditional on its size having the same distribution as the random  $m$-ary recursive tree.

We extend a known distributional recursive decomposition (which appears in~\cite{DOFI,Hofstad}) for the standard random recursive trees (the case $m=1$) to general $m$ as follows.  We denote the root of the tree by $v_0$ and its $m$ children emerging within the first step by $v_1,\ldots,v_m$. After $n-1$ steps (the highest index in the tree is then $n$), we decompose the tree into the~$m$ fringe trees rooted at $v_1,\ldots,v_m$, respectively (these are the trees rooted at $v_i$ with all the descendants of $v_i$ in the original tree, 
for $i=1, \ldots m$). Additionally, the remaining nodes of the tree form the 
$(m+1)$st tree of the decomposition. Hence, this $(m+1)$st tree consists of $v_0$ together with all the subtrees emerging from  $v_0$ after the first step. We call these $m+1$ trees the \textit{$m+1$ recursive parts of the tree}. 

To visualize the operations by which the decomposition is obtained, imagine
that the edges joining the root to the nodes $v_1, \ldots, v_m$ are simply removed. 
Figure~\ref{Fig:decomposition} illustrates a ternary recursive tree together with its decomposition into four trees.

\begin{figure}[ht]
\centering
\resizebox{\textwidth}{!}{
\begin{tikzpicture}[
vertex/.style={
circle,
draw,
minimum size=5mm,
inner sep=0pt,
font=\small
},
edge/.style={draw}
]

\begin{scope}

\node[vertex] (a1) at (0,0) {1};
\node[vertex] (a2)  at (-5,-1.5) {2};
\node[vertex] (a3)  at (-3,-1.5) {2};
\node[vertex] (a4)  at (-1,-1.5) {2};
\node[vertex] (a11) at ( 1,-1.5) {5};
\node[vertex] (a12) at ( 3,-1.5) {5};
\node[vertex] (a13) at ( 5,-1.5) {5};

\foreach \v in {2,3,4,11,12,13}
    \draw[edge] (a1)--(a\v);

\node[vertex] (a5) at (-5.7,-3) {3};
\node[vertex] (a6) at (-5.0,-3) {3};
\node[vertex] (a7) at (-4.3,-3) {3};

\foreach \v in {5,6,7}
    \draw[edge] (a2)--(a\v);

\node[vertex] (a8)  at (-3.7,-3) {4};
\node[vertex] (a9)  at (-3.0,-3) {4};
\node[vertex] (a10) at (-2.3,-3) {4};

\foreach \v in {8,9,10}
    \draw[edge] (a3)--(a\v);

\node[vertex] (a14) at (-1.7,-3) {6};
\node[vertex] (a15) at (-1.0,-3) {6};
\node[vertex] (a16) at (-0.3,-3) {6};

\foreach \v in {14,15,16}
    \draw[edge] (a4)--(a\v);

\node[vertex] (a17) at (-3.7,-4.5) {8};
\node[vertex] (a18) at (-3.0,-4.5) {8};
\node[vertex] (a19) at (-2.3,-4.5) {8};

\foreach \v in {17,18,19}
    \draw[edge] (a9)--(a\v);

\node[vertex] (a20) at (2.3,-3) {7};
\node[vertex] (a21) at (3.0,-3) {7};
\node[vertex] (a22) at (3.7,-3) {7};

\foreach \v in {20,21,22}
    \draw[edge] (a12)--(a\v);

\node at (0,-5.5) {\large The original tree};

\end{scope}

\node at (7,-2.3) {\Large $\Longrightarrow$};

\begin{scope}[xshift=14cm]

\node[vertex] (a1) at (0,0) {1};
\node[vertex] (a2)  at (-5,-1.5) {2};
\node[vertex] (a3)  at (-3,-1.5) {2};
\node[vertex] (a4)  at (-1,-1.5) {2};
\node[vertex] (a11) at ( 1,-1.5) {5};
\node[vertex] (a12) at ( 3,-1.5) {5};
\node[vertex] (a13) at ( 5,-1.5) {5};

\foreach \v in {11,12,13}
    \draw[edge] (a1)--(a\v);

\node[vertex] (a5) at (-5.7,-3) {3};
\node[vertex] (a6) at (-5.0,-3) {3};
\node[vertex] (a7) at (-4.3,-3) {3};

\foreach \v in {5,6,7}
    \draw[edge] (a2)--(a\v);

\node[vertex] (a8)  at (-3.7,-3) {4};
\node[vertex] (a9)  at (-3.0,-3) {4};
\node[vertex] (a10) at (-2.3,-3) {4};

\foreach \v in {8,9,10}
    \draw[edge] (a3)--(a\v);

\node[vertex] (a14) at (-1.7,-3) {6};
\node[vertex] (a15) at (-1.0,-3) {6};
\node[vertex] (a16) at (-0.3,-3) {6};

\foreach \v in {14,15,16}
    \draw[edge] (a4)--(a\v);

\node[vertex] (a17) at (-3.7,-4.5) {8};
\node[vertex] (a18) at (-3.0,-4.5) {8};
\node[vertex] (a19) at (-2.3,-4.5) {8};

\foreach \v in {17,18,19}
    \draw[edge] (a9)--(a\v);

\node[vertex] (a20) at (2.3,-3) {7};
\node[vertex] (a21) at (3.0,-3) {7};
\node[vertex] (a22) at (3.7,-3) {7};

\foreach \v in {20,21,22}
    \draw[edge] (a12)--(a\v);

\node at (0,-5.5) {\large The resulting forest};

\end{scope}

\end{tikzpicture}}
\caption{Decomposition of a ternary recursive tree into four
ternary recursive trees by deleting the edges connecting the root with any of the nodes labeled with 2}
\label{Fig:decomposition}
\end{figure}
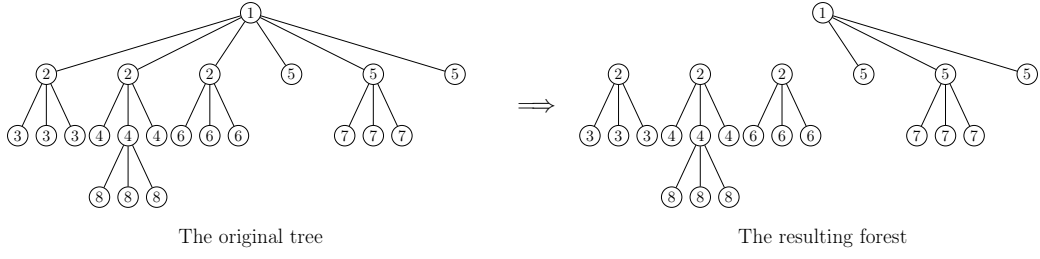

Strictly speaking, the recursive parts do not carry labels corresponding to the way 
$m$-ary recursive trees are labeled. Take for instance the third tree from left in the decomposition in Figure~\ref{Fig:decomposition}.
The nodes in this tree are labeled with 2 and 6. It would be a ternary recursive tree in the strict
sense of labeling, if the 2 is reduced to 1 and the 6 is reduced to 2. However,
this recursive part has all the structural properties of a ternary recursive tree of size~4: same shape, and consequently same topological values, such as height, Wiener index, etc., and occurs with the same probability.
  
More generally, each recursive part can be made to coincide with an $m$-ary recursive tree, occurring with the same probability and carrying the same labels, if the numbers in it are changed to be their
ranks within the set of numbers in the tree. The only aspect of the recursive 
parts that pertains to the analysis of topological indices is that they are {\em isomorphic}
to $m$-ary trees of the same sizes and occurring with the same probabilities.

The Wiener index $W_n$ can be derived from the Wiener indices of its $m+1$ recursive parts. To prepare for such a recursive distributional decomposition, let 
$L_i^{(n)}$ denote the number of steps performed within the subtree
rooted at $v_i$, for $i=1, \ldots, m+1$,   when the full tree is at age $n-1$, and collect these contributions in the row vector
$L^{(n)}=(L_1^{(n)},\ldots,L_{m+1}^{(n)})$. Set $I^{(n)}=(I_1^{(n)},\ldots,I_{m+1}^{(n)}):=L^{(n)}+(1,\ldots,1)$. This shift is necessary to be in accordance with the literature such as \cite{Neininger}, to which our subsequent setup reduces in the case $m=1$.

Let $I_i^{(n)}$ be the internal time  
spent on the construction of the $i$th part. 
In other words, it is the number of updates or additions in the $i$th recursive part.
Conditional on $I^{(n)}=(i_1,\ldots,i_{m+1})$, 
with $i_r\in\{1,\dots,n-1\}$ and $\sum_{r=1}^{m+1}i_r=n+m-1$, we have $m+1$ stochastically independent recursive parts that are distributed as random $m$-ary recursive trees after $i_1-1,\ldots,i_{m+1}-1$ steps, respectively. Furthermore, we use $N^{(r)}_{i_r}$ for the number of nodes within the $r$th recursive part after $i_r-1$ steps. Hence, 
$N^{(r)}_{i_r}$ is equal to
$1+m(i_r-1)=mi_r-(m-1)$. 

To analyze the distribution of $W_n$ asymptotically, we need $P_n$ as well. These two quantities are dependent. (Below, we derive their asymptotic correlation.) For this reason, we consider their joint distribution and hence the pair $(W_n,P_n)$. For $r=1,\dots,m+1$, let $(W_n^{(r)},P_n^{(r)})_{n\ge 0}$ be independent copies of $(W_n,P_n)_{n\ge 0}$. 

We start with the recursive distributional decomposition of the internal path length. 
The internal path length can be accumulated from the recursive parts---the $(m+1)$st 
part contributes its own internal path length. The internal path length 
of the $i$th recursive part, for $i=1, \ldots,m$ enters the picture with its value adjusted by the size of the part,
as the contribution of each node in the part adds an extra 1 upon hooking to the root
of the full tree. This gives rise to the distributional recurrence
\begin{align*} 
    P_n&\ \eqlaw \ P_{I_{m+1}^{(n)}}^{(m+1)}+\sum_{r=1}^{m}\Big( P_{I_r^{(n)}}^{(r)}+N^{(r)}_{I_r^{(n)}}\Big)\\
    & =\Big(\sum_{r=1}^{m+1}P_{I_r^{(n)}}^{(r)} \Big)+m(n-1)+1-N^{(m+1)}_{I_{m+1}^{(n)}},
\end{align*} 
The Wiener index of the full tree can be obtained from the 
recursive parts, too. Each part contributes its own Wiener index and adjustments
upon hooking to the root of the full tree. The parts are taken in pairs, and their respective 
sizes determine the amount of inter-pair adjustment.
According to this argument, the Wiener index for the full tree satisfies the
distributional recurrence
\begin{align*} 
    W_n &\ \eqlaw \ \sum_{r=1}^{m+1} W_{I_r^{(n)}}^{(r)} 
                    + \sum_{1 \le i < j \le m} 
                   \big(N_{I_i^{(n)}}^{(i)}P_{I_j^{(n)}}^{(j)}
                    + N_{I_j^{(n)}}^{(j)}P_{I_i^{(n)}}^{(i)}+2N_{I_i^{(n)}}^{(i)}N_{I_j^{(n)}}^{(j)}\big)\\ 
    &\qquad {}+\sum_{k=1}^m \big(N_{I_k^{(n)}}^{(k)}P_{I_{m+1}^{(n)}}^{(m+1)}+N_{I_{m+1}^{(n)}}^{(m+1)}P_{I_k^{(n)}}^{(k)}+N_{I_k^{(n)}}^{(k)}N_{I_{m+1}^{(n)}}^{(m+1)}\big). 
\end{align*} 
 Consequently, $W_n$ and $P_n$ satisfy the two-dimensional distributional recurrence
\begin{align*} 
    \begin{pmatrix} W_n \\ P_n \end{pmatrix}\eqlaw \sum_{r=1}^{m+1} A_r(n) \begin{pmatrix} W_{I_r^{(n)}}^{(r)} \\ P_{I_r^{(n)}}^{(r)} \end{pmatrix}
    +b(n),
\end{align*} 
with random $2\times 2$ matrices
$$
    A_r(n)=\begin{pmatrix} 1 & m\big(n-I_r^{(n)}\big) \\ 0 & 1 \end{pmatrix},$$
and a random vector
$ b(n)=\big(\begin{smallmatrix} b_1(n)\\ b_2(n)\end{smallmatrix}\big)$,
with the two components

\begin{align*} 
    b_1(n)&=2\sum_{1 \le i<j \le m} \big(m\big(I_i^{(n)}-1\big)+1\big)\big(m\big(I_j^{(n)}-1\big)+1\big)\\
    &\quad\quad~+\sum_{k=1}^m\big(m\big(I_k^{(n)}-1\big)+1\big)\big(m \big(I_{m+1}^{(n)}-1\big)+1\big),\\ 
    b_2(n)&=m(n-I_{m+1}^{(n)}). 
\end{align*} 

According to Proposition \ref{Prop:Wiener}, an asymptotic equivalent to the Wiener index is
$$
    \alpha_n:=\E[W_n]=m^2n^2\ln(n)\
          -m^2 \Big(1+\psi\Big(\frac{m+1}{m}\Big)\Big)n^2+o(n^2),$$
and according to Corollary~\ref{Cor:IPL},
an asymptotic equivalent to the internal path length is 
$$\gamma_n: =
       mn\ln(n)-m\psi\Big(\frac{m+1}{m}\Big)n+o(n).$$
\subsection{A limit law for the joint distribution of the Wiener index and the internal path length}
We work on the space of all centered probability measures on $\R^2$ with finite second moments denoted by 
\begin{align*}
    \mathcal{M}^2_{0,2}:=\big{\{}
              \mathcal{L}(X)\,|\,\E[\|X\|^2]<\infty, \E[X]=0\big{\}}.
\end{align*}
Furthermore, we define the Wasserstein-$\ell_2$ metric on the space $\mathcal{M}^2_{0,2}$
\begin{align*}
\ell_2(\mu,\nu):=\inf\big\{\E\left[\Big\Vert Y-Z\right\Vert^2\Big]^{1/2}:\mathcal{L}(Y)=\mu,\mathcal{L}(Z)=\nu\big\}.
\end{align*}
For more detailed information on $\ell_2$ in the present context, see \cite{Neininger2001}. 
\begin{theorem}\label{thm_1} 
    Let
    \begin{align*} 
        ( W_n^\ast, P_n^\ast )^T:=\Big(\frac {W_n- \E[W_n]} {n^2},  \frac{P_n-\E[P_n]}{n}\Big)^T.
    \end{align*} 
    Then, we have
    \begin{align*}
        \begin{pmatrix}
            W_n^\ast \\
            P_n^\ast
        \end{pmatrix} \overset{\ell_2}{\longrightarrow} \begin{pmatrix}
            W \\ P
        \end{pmatrix},
    \end{align*}
    where $\mathcal{L}(W,P)$ is the unique fixed point of the map $T:\mathcal{M}^2_{0,2} \to \mathcal{M}^2_{0,2}$ given by
    \begin{align*}
        T(\nu):=\mathcal{L}\Big(\sum_{r=1}^{m+1} A_r \begin{pmatrix}
            W^{(r)}\\
            P^{(r)}
        \end{pmatrix}+b\Big),
    \end{align*}
    with $A_r$ and $b$ given in \eqref{limit:A_r} and \eqref{limit:b}, where $(
        W^{(1)},
        P^{(1)})
   ,\dots,(
        W^{(m+1)},
        P^{(m+1)})$, $(V_1,\dots,V_{m+1})$ are independent and $\mathcal{L}
        ((W^{(r)},
        P^{(r)})^T)
   =\nu$, for all $1\le r\le m+1$. 
\end{theorem} 
A proof of this theorem is given in Section \ref{sec_proof_5.1}. Note that the $\ell_2$-convergence in Theorem \ref{thm_1} implies, firstly, convergence in distribution of the random vectors $( W_n^\ast, P_n^\ast )$, and secondly, convergence of the second (mixed) moments.  Hence, we obtain the following:
\begin{corollary}
We have
\begin{align*}
    \Var(W_n) \sim \sigma_W^2n^4, \qquad \Var(P_n) \sim \sigma_P^2n^2, \qquad \Cov(W_n,P_n) \sim \sigma_{WP}n^3,
\end{align*}
with constants $\sigma_W^2$,  $\sigma_P^2$,  $\sigma_{WP}$ given in Theorem \ref{thm:covariance-fixed-point} (see also Table \ref{table_1} for numerical values).
\end{corollary}

Moreover, we obtain a weak law of large numbers for the Wiener
index. 
\begin{corollary} 
We have 
$$\frac {W_n} {n^2 \ln (n)} \inprob   m^2.$$ 
\end{corollary} 
\subsubsection{The contraction method in $\R^d$ as a Hilbert space}
\label{sec:cm} 
To prove Theorem \ref{thm_1}, we employ the contraction method based on \cite{Neininger2001}. However, we take a more general point of view and work with weighted norms, see (\ref{lambda_norm}), instead of the Euclidean norms. This enables us to verify the contraction property needed for the contraction method, see Lemmas \ref{normshift} and \ref{contraction}. We consider the use of weighted norms as a technical novel part of our analysis.

Let $\langle \cdot,\cdot \rangle$ be some scalar-product on $\R^2$ (the case $\R^d$ is similar) inducing a norm $\Vert \cdot \Vert$, and let $A^\ast$ be the adjoint matrix of $A$, i.e. $\langle Ax,y\rangle=\langle x,A^\ast y\rangle$, for all $x,y \in \R^2$.
Let $(X_n)_{n \ge 0}$ denote a sequence of real-valued, centered, two-dimensional random vectors satisfying the distributional recurrence
\begin{align}\label{recurrence} 
    X_n \overset{d}{=} \sum_{r=1}^K A_r^{(n)} X_{I_r^{(n)}}^{(r)}+b^{(n)},
\end{align} 
for $n \ge n_0$, where $(X_n^{(1)})_{n \ge 0},\dots,(X_n^{(K)})_{n \ge 0}$ and $(A_1^{(n)},\dots,A_K^{(n)},b^{(n)})$ are independent, and $X_j^{(r)}$ is distributed as $X_j$ for all $r=1,\dots,K$ and $j \ge 0$. The coefficients $A_r^{(n)}$ are real-valued random $2 \times 2$ matrices, $b^{(n)}$ is a centered, real-valued $2$-dimensional row
vector, and $I^{(n)}=\big(I_1^{(n)},
\dots,I_K^{(n)}\big)$ is a vector of random integers in $\{0,\dots,n\}$, while $K$ and $n_0$ are constants. All these random quantities are assumed to be square-integrable. 

Furthermore, we assume that the following conditions hold: 
\begin{enumerate} 
    \item[(A)] $\big(A_1^{(n)},\dots,A_K^{(n)},b^{(n)}\big) \overset{\ell_2}{\longrightarrow} 
   (A_1,\dots,A_K,b)$, 
    \item[(B)] $\sum_{r=1}^K {\E\big[\Vert A_r^\ast A_r\Vert_{\op}\big]}<1$, 
    \item[(C)] $\sum_{r=1}^K {\E\Big[\eins_{\{I_r^{(n)}\le \ell\}\cup\{I_r^{(n)}=n\}} \big\Vert A_r^{(n),\ast}A_r^{(n)}\big\Vert_{\op}\Big]} \to 0$, as $n \to \infty$, for all constants $\ell \ge 0$. 
\end{enumerate} 

We can now state a theorem that is closely related
to \cite[Lemma 3.1, Theorem 4.1]{Neininger2001}. Since we do not use the operator norm associated with the standard Euclidean norm, for which the adjoint and transpose matrices are identical, we have to state conditions (B) and (C) 
with the adjoint matrix. A more general development of the contraction method in Hilbert spaces can be found in~\cite{drjane08}, which is not required in the present context. 
\begin{theorem}
Let $(X_n)_{n \ge 0}$ be a sequence satisfying \eqref{recurrence} with the properties mentioned above. Furthermore, let the conditions {\rm (A)--(C)} be satisfied. Then, we have $\ell_2(X_n,X) \longrightarrow 0$, where $\mathcal{L}(X)$ is the unique fixed point of the map $T:\mathcal{M}^2_{0,2} \to \mathcal{M}^2_{0,2}$ defined by
\begin{align*}
    T(\nu)=\mathcal{L}\Big(\sum_{r=1}^K A_r X^{(r)}+b \Big),
\end{align*}
where $(A_1,\dots,A_K,b),X^{(1)},\dots,X^{(K)}$ are independent and $\mathcal{L}(X^{(r)})=\nu$, for $r=1,\dots,K$.
\end{theorem}
\begin{proof}
Note that the map is well-defined since it preserves second moments by independence and centering, as $\E[b]=0$ by property (A). To show that the map $T$ has a unique fixed point, we must show that $T$ is a contraction with respect to the $\ell_2$-metric. 
Following the proof of \cite[Lemma 3.1]{Neininger2001}, 
for $\mu,\nu \in \mathcal{M}^2_{0,2}$, we obtain
\begin{align}
    \ell_2^2\big(T(\mu),T(\nu)\big) &\le \sum_{r=1}^K \E\big[\big\langle A_r\left(Y^{(r)}-Z^{(r)}\big),A_r\big(Y^{(r)}-Z^{(r)}\big)\big\rangle\right]\nonumber\\
    &=\sum_{r=1}^K \E\big[\big\langle Y^{(r)}-Z^{(r)},A_r^\ast A_r \big(Y^{(r)}-Z^{(r)}\big) \big\rangle\big] , \label{last:_step}
\end{align}
where $(Y^{(1)},Z^{(1)}),\dots,(Y^{(K)},Z^{(K)})$ are optimal couplings for $\mu$ and $\nu$ such that $(A_1,\dots,A_K,b),(Y^{(1)},Z^{(1)}),\dots,(Y^{(K)},Z^{(K)})$ are independent (note that this is always possible in the 
$\ell_2$-metric).  In (9) of~\cite{Neininger2001}  
the transpose matrix is used, but in our setting this is exactly the place where the adjoint matrix has to be used. Condition (B) ensures that we get $\ell_2(T(\mu), T(\nu)) \le c \ell_2(\mu,\nu)$, for some constant $c<1$, which shows the existence of a unique fixed point of 
$T$ in $\mathcal{M}^2_{0,2}$.

To prove $\ell_2(X_n,X) \to 0$, we can also follow the proof of~\cite[Theorem 4.1]{Neininger2001} and use the properties of $A_r^\ast$ and $A_r^{(n),\ast}$ at the instances where in the original proof the transpose matrix is used.
\end{proof}
Readers interested in the proof techniques are referred to \cite{Neininger2001}.
\subsubsection{A Pólya urn model for the internal time within the recursive parts}
\label{polya_urn}
We consider a Pólya urn consisting of $m+1$ colors. Let $J^{(n)}=(J_1^{(n)},\dots,J_{m+1}^{(n)})$ denote the urn composition after $n$ draws with initial composition $J^{(0)}=(1,\dots,1)$ and replacement matrix $R=mI$, where $I$ denotes the identity matrix. This implies that, if a ball of color $i \in \{1,\dots,m+1\}$ is chosen, it is returned to the urn  
alongside $m$ balls of the same color. See \cite{Mahmoud2009} for an exposition of Pólya urns. The following convergence result for P\'olya urns  is well known. 
 We use a version of P\'olya urns adapted to the need of the analysis 
of $m$-ary recursive trees.
\begin{lemma}\label{pol_conv}
   For $n \to \infty$, we have
\begin{align}\label{convergence:polya_urn}
\Big(\frac{J_1^{(n-2)}}{m(n-1)+1},\dots,\frac{J_{m+1}^{(n-2)}}{m(n-1)+1}\Big) \almostsure (V_1,\dots,V_{m+1}),
\end{align}
where $(V_1,\dots,V_{m+1})$ is \Dirichlet$(1/m,\dots,1/m)$-distributed. 
\end{lemma}
 \begin{proof}
     For the scaled process $\tilde{J}^{(n)}:=J^{(n)}/m$, we have $\tilde{J}^{(0)}=
     (1/m,\dots,1/m)$ and the replacement matrix $\tilde{R}=I$. The statement is then covered by the literature; see \cite[Corollary 1]{ath69} or \cite[Lemma 2.1]{knne14}.
 \end{proof} 
 \noindent To transfer Lemma \ref{pol_conv} to the internal time $I_i^{(n)}$ of the subtrees, note that we have $J_i^{(n-2)}=1+m(I_i^{(n)}-1)$, 
 since drawing a ball of color $i$ is equivalent to choosing a node in the $i$th subtree. The index shift occurs since there are only $(n-1)$ steps to obtain an $m$-ary recursive tree with the highest index $n$. 
 Therefore, using \eqref{convergence:polya_urn}, we obtain almost surely
\begin{align*}
    \Big(\frac{I_1^{(n)}-1}{n},\dots,\frac{I_{m+1}^{(n)}-1}{n}\Big)&=\Big(\frac{J_1^{(n-2)}-1}{mn},\dots,\frac{J_{m+1}^{(n-2)}-1}{mn}\Big)\to (V_1,\dots,V_{m+1}),
\end{align*}
which directly yields $(I_1^{(n)}/n,\dots,I_{m+1}^{(n)}/n) \to (V_1,\dots, V_{m+1})$ almost surely.
\subsubsection{Proof of Theorem \ref{thm_1}} \label{sec_proof_5.1}
The random vector $(W_n^\ast,P_n^\ast)^T$ satisfies recurrence \eqref{recurrence} with
\begin{align*} 
    A_r^{(n)}&=\begin{pmatrix} 1/n^2 & 0\\ 0 & 1/n \end{pmatrix} A_r(n) \begin{pmatrix} \big(I_r^{(n)}\big)^2 & 0\\ 0 & I_r^{(n)} \end{pmatrix}\\
    &= \begin{pmatrix} \big(I_r^{(n)}/n\big)^2 & mI_r^{(n)}\big(n-I_r^{(n)}\big)/n^2 \\ 0 & I_r^{(n)}/n \end{pmatrix}. 
\end{align*}
The normalized toll term $b^{(n)}$ is given by
\begin{align*} 
    \begin{pmatrix} b_1^{(n)}\\ b_2^{(n)} \end{pmatrix}&=\begin{pmatrix} 1/n^2 & 0 \\ 0 & 1/n \end{pmatrix}\left(\sum_{r=1}^{m+1}\begin{pmatrix} 1 & m\big(n-I_r^{(n)}\big) \\ 0 & 1 \end{pmatrix}\begin{pmatrix} \alpha_{I_r^{(n)}}\\ \gamma_{I_r^{(n)}} \end{pmatrix}\right.\\
   &\left. 
    ~\qquad\qquad\qquad\qquad-\begin{pmatrix} \alpha_n \\ \gamma_n \end{pmatrix}+\begin{pmatrix} b_1(n)\\ b_2(n) \end{pmatrix}\right). 
\end{align*}
To verify condition (A), we first identify the limiting matrices $A_r$ and the vector $b$.
\begin{lemma} 
    We have
    \begin{align}\label{limit:A_r} 
        A_r^{(n)} \overset{\ell_2}{\longrightarrow} \begin{pmatrix} V_r^2 
             & mV_r(1-V_r)\\ 0 & V_r \end{pmatrix}=: A_r,
    \end{align} 
    and
    \begin{align}\label{limit:b} 
        \begin{pmatrix} b_1^{(n)}\\ b_2^{(n)}\end{pmatrix} \overset{\ell_2}{\longrightarrow} \begin{pmatrix} b_1\\ b_2 \end{pmatrix}=: b,
    \end{align} 
    with 
$$    b_1=m^2\Big(\sum_{r=1}^{m+1} V_r(\ln(V_r)-V_r)+\sum_{r=1}^m V_r(1-V_r)+1\Big),$$
    and 
$$b_2=m\Big(\sum_{r=1}^{m+1} V_r\ln(V_r)+1-V_{m+1}\Big),$$
    where $(V_1,\dots,V_{m+1})$ is \Dirichlet$(1/m,\dots,1/m)$ distributed.
\end{lemma} 
\begin{proof} 
To simplify notation, let $I_r$ denote $I_r^{(n)}$. Expanding the equation for $b^{(n)}$ yields
    \begin{align} \label{final_term}
        b_1^{(n)}=\frac{1}{n^2}\Big(\sum_{r=1}^{m+1} \big( \alpha_{I_r}+m\left(n-I_r\big)\gamma_{I_r}\right)-\alpha_n+b_1(n)\Big). 
    \end{align} 
Applying the asymptotic expansions of $\alpha_n$ and $\gamma_n$ and setting $0\ln(0):=0$, we get
\begin{align*} 
  \frac{1}{n^2}\sum_{r=1}^{m+1} \alpha_{I_r}
        &=\frac{1}{n^2}\Big(\sum_{r=1}^{m+1}\Big(m^2I_r^2\ln(I_r)-m^2\Big(1+                   
                     \psi\Big(\frac{m+1}{m}\Big)\Big)I_r^2\Big)+o(n^2) \Big)\\ 
       &= m^2\sum_{r=1}^{m+1}\Big(\Big(\frac{I_r}{n}\Big)^2\ln(I_r)
           -\Big(1+\psi\Big(\frac{m+1}{m}\Big)\Big)\Big(\frac{I_r}{n}\Big)^2\Big)+o(1),
\end{align*} 
and
\begin{align*} 
 &\frac{1}{n^2}\sum_{r=1}^{m+1} m(n-I_r) \gamma_{I_r}\\
  & \qquad =m^2\sum_{r=1}^{m+1}\Big(\frac{(n-I_r)I_r}{n^2}\ln(I_r)-\psi\Big(\frac{m+1}{m}\Big) \frac{(n-I_r)I_r}{n^2}\Big)+o(1). 
\end{align*} 
Combining these results, we obtain
\begin{align*} 
 \lefteqn{   \frac{1}{n^2}\sum_{r=1}^{m+1} \Big(\alpha_{I_r}+m(n-I_r)\gamma_{I_r}\Big)}
     \\
     &\qquad =m^2\sum_{r=1}^{m+1} \Big(\frac{I_r}{n} \ln(I_r)-\psi\Big(\frac{m+1}{m}\Big)\frac{I_r}{n}-\Big(\frac{I_r}{n}\Big)^2\Big)+o(1), 
\end{align*} 
which implies
\begin{align*} 
 \lefteqn{   \frac{1}{n^2}\Big(\sum_{r=1}^{m+1} \Big(\alpha_{I_r}+m(n-I_r)\gamma_{I_r}\Big)-\alpha_n\Big)}\\
 &\qquad =m^2\Big(\sum_{r=1}^{m+1}  
        \Big(\frac{I_r}{n}\Big(\ln\Big(\frac{I_r}{n}\Big)-\frac{I_r}{n}\Big)\Big)+1\Big)+o(1). 
\end{align*} 
For the summand $b_1(n)$ in (\ref{final_term}), using $\sum_{r=1}^{m+1} I_r=n+m-1$, we have
\begin{align*} 
    \frac{b_1(n)}{n^2}&=m^2\Big(2\sum_{1 \le i<j \le m}\frac{I_i I_j}{n^2}+\sum_{r=1}^m \frac{I_r I_{m+1}}{n^2}\Big)+o(1)\\ 
    &=m^2\Big(\Big(\sum_{r=1}^m \frac{I_r}{n}\Big)^2-\sum_{r=1}^m \Big(\frac{I_r}{n}\Big)^2+\frac{I_{m+1}}{n}\sum_{r=1}^m \frac{I_r}{n}\Big)+o(1)\\ 
    &=m^2\sum_{r=1}^{m}\Big(\frac{n+m-1}{n} \frac{I_r}{n}-\Big(\frac{I_r}{n}\Big)^2\Big)+o(1)\\ 
    &=m^2\sum_{r=1}^{m} \Big(\frac{I_r}{n}\Big(\frac{n+m-1}{n}-\frac{I_r}{n}\Big)\Big)+o(1). 
\end{align*} 
The second component $b_2^{(n)}$ is computed as
\begin{align*} 
    b_2^{(n)}&=\frac{1}{n}\Big(\sum_{r=1}^{m+1} mI_r\ln(I_r)-mn\ln(n)+m(n-I_{m+1})+o(n) \Big)\\ 
    &=m\Big(\sum_{r=1}^{m+1} \frac{I_r}{n}\ln\Big(\frac{I_r}{n}\Big)+1-\frac{I_{m+1}}{n}\Big)+o(1). 
\end{align*} 
Since $(I_1/n,\dots,I_{m+1}/n)$ converges almost surely to $(V_1,\dots,V_{m+1})$,
 where $(V_1,\dots,V_{m+1})$ follows a \Dirichlet$(1/m,\dots,1/m)$ distribution, see Section \ref{polya_urn}, it follows that
\begin{align*} 
    A_r^{(n)} \overset{\ell_2}{\longrightarrow} \begin{pmatrix} V_r^2 & mV_r(1-V_r)\\ 0 & V_r \end{pmatrix}=:A_r,
\end{align*} 
as well as
\begin{align*} 
   b_1^{(n)} \overset{\ell_2}{\longrightarrow} m^2\Big(\sum_{r=1}^{m+1} V_r \big(\ln(V_r)-V_r\big)+\sum_{r=1}^m V_r (1-V_r)+1\Big),
\end{align*} 
and
\begin{align*} 
    b_2^{(n)}\overset{\ell_2}{\longrightarrow}m\Big(\sum_{r=1}^{m+1} V_r \ln(V_r)+1-V_{m+1}\Big), 
\end{align*} 
by the dominated convergence theorem, since $x \mapsto x\ln(x)$ is bounded on $[0,1]$ and all entries of $A_r^{(n)}$ are uniformly bounded. 
\end{proof} 
To verify the contraction condition (B), we use a weighted norm 
\begin{align}\label{lambda_norm}
\Vert x \Vert_\lambda:=\sqrt{x_1^2+\lambda x_2^2}\, ,
\end{align}
with $\lambda>0$. To perform the computations, we 
need the following lemma to relate the operator norm with respect to $\Vert \cdot \Vert_\lambda$ to the standard operator norm. 
In what follows, we denote the operator norm with respect to $\Vert \cdot \Vert_\lambda$ by $\Vert \cdot \Vert_{\op,\lambda}$.
\begin{lemma}\label{normshift} 
  We have $$\Vert A_r^\ast A_r \Vert_{\op,\lambda}=\big\Vert A_r^{[\lambda]}\big\Vert_{\op}^2,$$ where $\Vert \cdot \Vert_{\op}$ denotes the operator norm with respect to the standard Euclidean norm and
  \begin{align*} 
    A_r^{[\lambda]}=\begin{pmatrix} V_r^2 & mV_r(1-V_r)/\sqrt{\lambda}\, \\ 0 & V_r \end{pmatrix}. 
  \end{align*} 
\end{lemma} 
\begin{proof} 
    Consider the matrices 
    $$G_\lambda=\begin{pmatrix} 1 & 0 \\ 0 & \lambda \end{pmatrix}\quad\mbox{and}\quad D_\lambda=G_\lambda^{1/2}=\begin{pmatrix} 1 & 0 \\ 0 & \sqrt{\lambda} \end{pmatrix}.$$ Let $\Vert \cdot \Vert$ be the standard Euclidean norm. Then, for all $x \in \R^2$, we have    
    $\Vert x \Vert_\lambda=\Vert D_\lambda x\Vert$. For any matrix $B$, this implies
    \begin{align*} 
        \Vert B \Vert_{\op,\lambda}=\sup_{x \neq 0} \frac{\Vert Bx \Vert_\lambda}{\Vert x \Vert_\lambda}=\sup_{x \neq 0}\frac{\Vert D_\lambda Bx\Vert}{\Vert D_\lambda x \Vert}=\sup_{y \neq 0}\frac{\Vert D_\lambda B D_\lambda^{-1}y\Vert}{\Vert y \Vert}=\Vert D_\lambda B D_\lambda^{-1}\Vert_{\op},
    \end{align*} 
    where $y=D_\lambda x$. The adjoint matrix $A_r^\ast$ of $A_r$ with respect to $\Vert \cdot \Vert_\lambda$ is given by $A_r^\ast=G_\lambda^{-1} A_r^T G_\lambda$, where $A^T$ is the transpose. Thus, $D_\lambda A_r^\ast D_\lambda^{-1}=D_\lambda^{-1} A_r^T D_\lambda=(A_r^{[\lambda]})^T$, since $A_r^{[\lambda]}=D_\lambda A_r D_\lambda^{-1}$. Combining these results, we obtain
    \begin{align*} 
        D_\lambda A_r^\ast A_r D_\lambda^{-1}=(D_\lambda A_r^\ast D_\lambda^{-1})(D_\lambda A_r D_\lambda^{-1})=(A_r^{[\lambda]})^T A_r^{[\lambda]} .
    \end{align*} 
 Therefore, $\Vert A_r^\ast A_r \Vert_{\op,\lambda}=\Vert D_\lambda A_r^\ast A_r D_\lambda^{-1}\Vert_{\op}=\Vert (A_r^{[\lambda]})^T A_r^{[\lambda]}\Vert_{\op}$. The assertion follows from the property that for all real matrices $M$, $\Vert M^T M \Vert_{\op}=\Vert M \Vert_{\op}^2$. 
\end{proof} 
Now we are able to show that there exists a weighted norm $\Vert \cdot \Vert_\lambda$ such that the contraction condition (B) is fulfilled.
\begin{lemma}\label{contraction} 
    There exists a $\lambda>0$ such that
$$ \mathbb{E}\Big[\sum_{r=1}^{m+1} \Vert A_r^\ast A_r \Vert_{\op,\lambda}\Big]<1. $$
\end{lemma} 
\begin{proof} 
Using Lemma \ref{normshift}, we have 
\begin{align*} 
        \mathbb{E}\Big[\sum_{r=1}^{m+1} \Vert A_r^\ast A_r \Vert_{\op,\lambda}\Big]=\sum_{r=1}^{m+1} \E\Big[\Big\Vert A_r^{[\lambda]} \Big\Vert_{\op}^2\Big]. 
\end{align*} 
As $\lambda \to \infty$, we observe that 
\begin{align*} 
    A_r^{[\lambda]} \almostsure\begin{pmatrix} V_r^2 & 0 \\ 0 & V_r \end{pmatrix}. 
\end{align*} 
Since the map $M \mapsto \Vert M \Vert_{\op}^2$ is continuous, it follows that $\Vert A_r^{[\lambda]}\Vert_{\op}^2 \to V_r^2$ almost surely. Using the Dominated Convergence Theorem, we get 
\begin{align*} 
    \mathbb{E}\Big[\big\Vert A_r^{[\lambda]}\big\Vert_{\op}^2\Big] \to \mathbb{E}[V_r^2]=\frac{1}{2m+1}, 
\end{align*} 
where we use the fact that $V_i$ is Beta$(1/m,1)$-distributed, as the marginal of a \Dirichlet$(1/m,\dots,1/m)$ distribution. Thus, there exists $\lambda_0$ such that for all $\lambda \ge \lambda_0$, we have
\begin{align*} 
    \mathbb{E}\big[\Vert A_r^{[\lambda]}\Vert_{\op}^2\big] \le \delta<\frac{1}{m+1}.
\end{align*} 
Since the $V_i$'s are identically distributed, we obtain
\begin{align*} 
    \mathbb{E}\Big[\sum_{r=1}^{m+1}\big\Vert A_r^{[\lambda]}\big\Vert_{\op}^2\Big] \le (m+1)\delta<1. 
\end{align*} 
\end{proof}
Last but not least, we have to show condition (C) for the weighted norm $\Vert \cdot \Vert_\lambda$.
\begin{lemma} 
    For $r=1,\dots,m+1$ and all $\ell \in \N_0$, we have
    \begin{align}\label{technical_condition} 
    \mathbb{E}\Big[ 1_{\{I_r^{(n)} \le \ell\}\cup\{I_r^{(n)}=n\}}\Vert A_r^{(n),\ast} A_r^{(n)} \Vert_{\op,\lambda}\Big] \to 0. 
\end{align} 
\end{lemma} 
\begin{proof} 
Applying Lemma \ref{normshift} again, we obtain 
\begin{align*} 
  \Vert A_r^{(n),\ast} A_r^{(n)}\Vert_{\op,\lambda} &= \left\Vert \begin{pmatrix} \big(I_r^{(n)}/n\big)^2 & mI_r^{(n)}(n-I_r^{(n)})/(\sqrt{\lambda\, } n^2) \\ 0 & I_r^{(n)}/n \end{pmatrix} \right\Vert_{\op}^2 \\
    &\le 2+\frac{m^2}{\lambda}. 
\end{align*} 
In the last step, we use $0 \le I_r^{(n)}/n \le 1$, and the fact that we have $\Vert A \Vert_{\op} \le \Vert A \Vert_F$ for all matrices $A$, where $\Vert \cdot \Vert_F$ denotes the Frobenius norm.  Furthermore, by weak convergence, we obtain
\begin{align*} 
    \limsup_{n \to \infty}\prob\big(I_r^{(n)} \le \ell \big) \le \limsup_{n \to \infty} \prob
    \Big(\frac{I_r^{(n)}}{n} \le \varepsilon\Big) = \mathbb{P}(V_r\le \varepsilon)=\sqrt[m]{\varepsilon}, 
\end{align*} 
for all $0<\varepsilon<1$. Since $\varepsilon$ can be chosen arbitrarily small, we obtain that $\mathbb{P}(I_r^{(n)}\le \ell)\to 0$ for $n \to \infty$. Consequently, \eqref{technical_condition} follows, since we have $\mathbb{P}(I_r^{(n)}=n)=0$. 
This completes the proof of Theorem \ref{thm_1}.
\end{proof} 
\subsubsection{The covariance matrix}
In a second step, we can compute the covariance matrix of $(W,P)$ to obtain the leading terms of $\Var(W_n)$, $\Var(P_n)$ and $\Cov(W_n,P_n)$. 
\begin{theorem}\label{thm:covariance-fixed-point}
Let
$\mathcal{L}((W,P)^T)$
be the unique fixed point in $\mathcal{M}^2_{0,2}$ of the map~$T$ from Theorem \ref{thm_1}.
\[
{\mathbf \Sigma}:=\Cov\!\begin{pmatrix}W\\ P\end{pmatrix}
=
\begin{pmatrix}
\sigma_W^2 & \sigma_{WP}\\
\sigma_{WP} & \sigma_P^2
\end{pmatrix}.
\]
Then the entries are given by
\begin{align*}
\sigma_P^2
&=
\frac{2m+1}{m}\,\Var(b_2),
\\[1ex]
\sigma_{WP}
&=
\frac{m+1}{2}\,\Var(b_2)
+
\frac{3m+1}{2m}\,\Cov(b_1,b_2),
\\[1ex]
\sigma_W^2
&=
\frac{m(m+1)}{3}\,\Var(b_2)
+
\frac{m+1}{3}\,\Cov(b_1,b_2)
+
\frac{4m+1}{3m}\,\Var(b_1).
\end{align*}
\end{theorem}
The quantities $\Var(b_1), \Var(b_2), \Cov(b_1,b_2)$ are explicitly computed in 
Proposition \ref{prop:explicit-b-covariances}, 
see the appendix,
where also a proof of Theorem \ref{thm:covariance-fixed-point} is given. The computations in the appendix allow us to compute the values $\sigma_P^2$, 
$\sigma_{WP}$ and $\sigma_W^2$ numerically. The first 10 of these values are given in Table \ref{table_1}.
\begin{table}[ht]
\centering
\caption{Numerical values of $\sigma_P^2$, $\sigma_{WP}$, and $\sigma_W^2$, for $m=1,\dots,10$.}\label{table_1}
\begin{tabular}{|r||c|c|c|}
\hline  \hline
$m$ & $\sigma_P^2$ & $\sigma_{WP}$ & $\sigma_W^2$ \\
\hline \hline
1  & 0.35507 & 0.10507 & 0.07729 \\  \hline
2  & 1.59412 & 1.05492 & 1.50031 \\ \hline
3  & 3.63963 & 3.68673 & 7.87806 \\ \hline
4  & 6.44273 & 8.70427 & 24.89798 \\ \hline
5  & 9.98224 & 16.76724 & 60.16197 \\ \hline
6  & 14.24781 & 28.51985 & 123.14688 \\ \hline
7  & 19.23390 & 44.59980 & 225.19265 \\ \hline
8  & 24.93724 & 65.64169 & 379.49796 \\ \hline
9  & 31.35580 & 92.27853 & 601.11842 \\ \hline
10 & 38.48827 & 125.14241 & 906.96567 \\ \hline
\hline
\end{tabular}
\end{table}
\section*{Acknowledgment}
Some of this research was conducted while the fourth author was on sabbatical leave at The Catholic University of America  in Fall 2025. That author thanks the Department of Mathematics and Statistics there for their gracious reception and hospitality and for an environment conducive to collaborative research. 

The authors sincerely thank Colin Desmarias for carefully reading a preliminary draft and suggesting several improvements. He drew our attention to the relevance of certain references.
\section{Tools and computational resource disclosure}
At the recommendation of the Leiden Declaration on Artificial Intelligence and Mathematics \cite{Leiden26}, 
we disclose the extent to which AI was used in the writing and production of this manuscript.

Parts of the text were linguistically refined with the assistance of AI, specifically ChatGPT-5.4. The computation of the constants in Theorem \ref{thm:covariance-fixed-point}, stated explicitly in the appendix, is elementary and consists mainly of standard integrals that can be evaluated either by hand using classical integration techniques or with a symbolic or numerical computing environment such as Maple. Because a large number of such terms arise, we used ChatGPT-5.4 to help organize the computations. We carefully verified the resulting expressions. The numerical values reported in Table \ref{table_1} are obtained using a Python program based on the formulas given in the appendix. 
 
\appendix
\section{Computation of the Covariance Matrix} 
\label{app_A}
We start by proving Theorem \ref{thm:covariance-fixed-point}.
\begin{proof}[Proof of Theorem \ref{thm:covariance-fixed-point}]
Using $\E[W]=\E[P]=\E[b_1]=\E[b_2]=0$ and writing $Y^{(r)}:=(   W^{(r)},
    P^{(r)})^T$, taking covariances in the fixed-point equation yields
\begin{align*}
    \Sigma=\E\left[\left(\sum_{r=1}^{m+1} A_rY^{(r)}+b\right)\left(\sum_{r=1}^{m+1} A_rY^{(r)}+b\right)^T\right]
\end{align*}
which gives us
\begin{align*}
    \Sigma&=\sum_{r=1}^{m+1}\sum_{s=1}^{m+1} \E\left[ A_r Y^{(r)} \left(Y^{(s)}\right)^T A_s^T\right]+\sum_{r=1}^{m+1} \E\left[A_rY^{(r)}b^T\right]\\
    &\quad +\sum_{s=1}^{m+1} \E\left[b\left(Y^{(s)}\right)^T A_s^T\right]+\E\left[bb^T\right].
\end{align*}
The second and third summands vanish since $(A_1,\dots,A_{m+1},b)$ is independent of $Y^{(r)}$ for all $r=1,\dots,m+1$ and we have $\E\left[Y^{(r)}\right]=0$ while $\E\left[bb^T\right]=\Cov(b)$ since $\E[b]=0$. Furthermore, we have
\begin{align*}
    \E\left[A_rY^{(r)}\left(Y^{(r)}\right)^T A_r^T\,\middle|\,A_r\right]=A_r\Sigma A_r^T
\end{align*}
and
\begin{align*}
    \E\left[A_rY^{(r)}\left(Y^{(s)}\right)^T A_s^T\,\middle|\,(A_1,\dots,A_{m+1})\right]=0,
\end{align*}
once again using $\E\left[Y^{(r)}\right]=0$ and the independence of $(A_1,\dots,A_{m+1})$ and $Y^{(i)}$. Putting these estimates together gives us
\begin{align}\label{Sigma}
    \Sigma=\sum_{r=1}^{m+1} \E\left[A_r \Sigma A_r^T\right]+\Cov(b).
\end{align}
To compute the first term in \eqref{Sigma}, we write
\[
\Sigma=
\begin{pmatrix}
\sigma_W^2 & \sigma_{WP}\\
\sigma_{WP} & \sigma_P^2
\end{pmatrix}
\]
to obtain
\[
A_r\Sigma A_r^T
=
\begin{pmatrix}
\Upsilon_{11} & \Upsilon_{12}\\
 \Upsilon_{12} &  \Upsilon_{22}
\end{pmatrix}\]
with 
\begin{align*}
\Upsilon_{11}&:=V_r^4\sigma_W^2
+2mV_r^3(1-V_r)\sigma_{WP}
+m^2V_r^2(1-V_r)^2\sigma_P^2,\\
\Upsilon_{12}&:=V_r^3\sigma_{WP}+mV_r^2(1-V_r)\sigma_P^2,\\
 \Upsilon_{22}&:=V_r^2\sigma_P^2.
 \end{align*}

Hence, we get
\begin{align}
\sigma_P^2
&=
\sum_{r=1}^{m+1}\E[V_r^2]\,\sigma_P^2+\Var(b_2),
\label{eq:cov-fixed-point-P}
\\[1ex]
\sigma_{WP}
&=
\sum_{r=1}^{m+1}\E[V_r^3]\,\sigma_{WP}
+\sum_{r=1}^{m+1}\E[mV_r^2(1-V_r)]\,\sigma_P^2
+\Cov(b_1,b_2),
\label{eq:cov-fixed-point-WP}
\\[1ex]
\sigma_W^2
&=
\sum_{r=1}^{m+1}\E[V_r^4]\,\sigma_W^2
+2\sum_{r=1}^{m+1}\E[mV_r^3(1-V_r)]\,\sigma_{WP}
\notag\\
&\qquad
+\sum_{r=1}^{m+1}\E[m^2V_r^2(1-V_r)^2]\,\sigma_P^2
+\Var(b_1).
\label{eq:cov-fixed-point-W}
\end{align}
We now compute the required moments. Since
\[
(V_1,\dots,V_{m+1})\sim \Dir\!\left(\frac1m,\dots,\frac1m\right),
\]
every marginal $V_r$, for $r=1,\dots,m+1$, has Beta distribution
\[
V_r\sim \Beta\!\left(\frac1m,1\right).
\]
Therefore, for each integer $k\ge1$,
\[
\E\left[V_r^k\right]
=
\frac{1}{km+1}.
\]
In particular, we have
\[
\E\left[V_r^2\right]=\frac{1}{2m+1},
\qquad
\E\left[V_r^3\right]=\frac{1}{3m+1},
\qquad
\E\left[V_r^4\right]=\frac{1}{4m+1},
\]
which gives us
\[
\sum_{r=1}^{m+1}\E\left[V_r^2\right]=\frac{m+1}{2m+1},
\quad
\sum_{r=1}^{m+1}\E\left[V_r^3\right]=\frac{m+1}{3m+1},
\quad
\sum_{r=1}^{m+1}\E\left[V_r^4\right]=\frac{m+1}{4m+1}.
\]
Furthermore, we have
\begin{align*}
\E\left[V_r^2(1-V_r)\right]
&=
\E\left[V_r^2\right]-\E\left[V_r^3\right]
=
\frac{1}{2m+1}-\frac{1}{3m+1}\\
&=\frac{m}{(2m+1)(3m+1)},
\end{align*}
hence we obtain
\[
\sum_{r=1}^{m+1}\E\left[mV_r^2(1-V_r)\right]
=
\frac{m^2(m+1)}{(2m+1)(3m+1)}.
\]
Likewise, we can compute
\begin{align*}
\E\left[V_r^3(1-V_r)\right]
&=
\E\left[V_r^3\right]-\E\left[V_r^4\right]
=
\frac{1}{3m+1}-\frac{1}{4m+1}\\
&=\frac{m}{(3m+1)(4m+1)},
\end{align*}
and thus get
\[
\sum_{r=1}^{m+1}\E\left[mV_r^3(1-V_r)\right]
=
\frac{m^2(m+1)}{(3m+1)(4m+1)}.
\]
Finally, we have
\begin{align*}
\E\left[V_r^2(1-V_r)^2\right]
&=
\E\left[V_r^2\right]-2\E\left[V_r^3\right]+\E\left[V_r^4\right]\\
&=\frac{1}{2m+1}-\frac{2}{3m+1}+\frac{1}{4m+1}\\
&=\frac{2m^2}{(2m+1)(3m+1)(4m+1)},
\end{align*}
which implies
\[
\sum_{r=1}^{m+1}\E[m^2V_r^2(1-V_r)^2]
=
\frac{2m^4(m+1)}{(2m+1)(3m+1)(4m+1)}.
\]
Substituting these into \eqref{eq:cov-fixed-point-P}, we obtain
\[
\sigma_P^2
=
\frac{m+1}{2m+1}\sigma_P^2+\Var(b_2),
\]
which gives us
\[
\sigma_P^2=\frac{2m+1}{m}\Var(b_2).
\]
Next, \eqref{eq:cov-fixed-point-WP} has the form
\[
\sigma_{WP}
=
\frac{m+1}{3m+1}\sigma_{WP}
+
\frac{m^2(m+1)}{(2m+1)(3m+1)}\sigma_P^2
+
\Cov(b_1,b_2),
\]
which gives us
\[
\frac{2m}{3m+1}\sigma_{WP}
=
\frac{m^2(m+1)}{(2m+1)(3m+1)}\sigma_P^2
+
\Cov(b_1,b_2)
\]
and therefore
\[
\sigma_{WP}
=
\frac{m(m+1)}{2(2m+1)}\sigma_P^2
+
\frac{3m+1}{2m}\Cov(b_1,b_2).
\]
If we substitute $\sigma_P^2=(2m+1)\Var(b_2)/m$ into the formula for $\sigma_{WP}$, we obtain
\[
\sigma_{WP}
=
\frac{m+1}{2}\Var(b_2)
+
\frac{3m+1}{2m}\Cov(b_1,b_2).
\]
Finally, \eqref{eq:cov-fixed-point-W} yields
\begin{align*}
\sigma_W^2
&=
\frac{m+1}{4m+1}\sigma_W^2
+
2\frac{m^2(m+1)}{(3m+1)(4m+1)}\sigma_{WP}\\
&\quad+
\frac{2m^4(m+1)}{(2m+1)(3m+1)(4m+1)}\sigma_P^2
+
\Var(b_1).
\end{align*}
Hence, we get
\begin{align*}
\frac{3m}{4m+1}\sigma_W^2
&=
2\frac{m^2(m+1)}{(3m+1)(4m+1)}\sigma_{WP}\\
&\quad+
\frac{2m^4(m+1)}{(2m+1)(3m+1)(4m+1)}\sigma_P^2
+
\Var(b_1).
\end{align*}
and obtain
\[
\sigma_W^2
=
\frac{2m(m+1)}{3(3m+1)}\sigma_{WP}
+
\frac{2m^3(m+1)}{3(2m+1)(3m+1)}\sigma_P^2
+
\frac{4m+1}{3m}\Var(b_1).
\]
Inserting the expressions for $\sigma_P^2$ and $\sigma_{WP}$ into the formula for $\sigma_W^2$ gives us
\begin{align*}
\sigma_W^2
&=
\frac{2m(m+1)}{3(3m+1)}
\left(
\frac{m+1}{2}\Var(b_2)
+
\frac{3m+1}{2m}\Cov(b_1,b_2)
\right)\\
&+
\frac{2m^3(m+1)}{3(2m+1)(3m+1)}
\cdot
\frac{2m+1}{m}\Var(b_2)
+
\frac{4m+1}{3m}\Var(b_1),
\end{align*}
which simplifies to
\[
\sigma_W^2
=
\frac{m(m+1)}{3}\,\Var(b_2)
+
\frac{m+1}{3}\,\Cov(b_1,b_2)
+
\frac{4m+1}{3m}\,\Var(b_1),
\]
\end{proof}
In a last step, we compute the covariance matrix of $b$ in explicit terms of the digamma function $\psi$ and the trigamma function $\psi_1$.
\begin{proposition}\label{prop:explicit-b-covariances}
Let
\[
(V_1,\dots,V_{m+1})\sim \Dir\!\left(\frac1m,\dots,\frac1m\right),
\]
and let
\[
b_1
=
m^2\left(
\sum_{r=1}^{m+1}V_r\ln(V_r)
-\sum_{r=1}^{m+1}V_r^2
+\sum_{r=1}^{m}V_r(1-V_r)
+1
\right),
\]
\[
b_2
=
m\left(
\sum_{r=1}^{m+1}V_r\ln(V_r)+1-V_{m+1}
\right).
\]
Further, set
\[
S:=\sum_{r=1}^{m+1}V_r\ln(V_r),
\qquad
T:=\sum_{r=1}^{m}V_r^2,
\qquad
U:=V_{m+1}.
\]
Then
\begin{align*}
\Var(b_1)
&=
m^4\Bigl(
\Var(S)+\Var(U)+\Var(U^2)+4\Var(T)-2\Cov(S,U)
\\[-0.2ex]
&\hspace{4em}
-2\Cov(S,U^2)-4\Cov(S,T)+2\Cov(U,U^2)\\[-0.2ex]
&\hspace{4em}
+4\Cov(U,T)+4\Cov(U^2,T)
\Bigr),
\\[1ex]
\Var(b_2)
&=
m^2\Bigl(
\Var(S)+\Var(U)-2\Cov(S,U)
\Bigr),
\\[1ex]
\Cov(b_1,b_2)
&=
m^3\Bigl(
\Var(S)+\Var(U)-2\Cov(S,U)-\Cov(S,U^2)
\\[-0.2ex]
&\hspace{4.8em}
+\Cov(U,U^2)-2\Cov(S,T)+2\Cov(U,T)
\Bigr).
\end{align*}
Moreover,
\begin{align*}
\E[S] &= \psi\!\left(\frac{m+1}{m}\right) -\psi\!\left(\frac{2m+1}{m}\right),\\ 
\E[T] &= \frac{m}{2m+1},\\
\E[U] &= \frac{1}{m+1}, \\
\Var(U) &=\frac{m^2}{(m+1)^2(2m+1)},\\
             \Cov(U,T)&=\frac{m}{(2m+1)(3m+1)} -\frac{m}{(m+1)(2m+1)}.
\end{align*}
The remaining terms can be reduced to the following Dirichlet moments:
\begin{align*}
\Var(S)
&=
(m+1)\E\bigl[V_1^2(\ln(V_1))^2\bigr]
+m(m+1)\E\bigl[V_1V_2\ln(V_1)\ln(V_2)\bigr]
-\E[S]^2,
\\[1ex]
\Cov(S,U)&=
m\,\E\!\left[V_1 \ln(V_1)\,U\right]+\E\!\left[U^2 \ln(U)\right]-\E[S]\E[U],
\\[1ex]
\Var(T)
&=
m\E\bigl[V_1^4\bigr]
+m(m-1)\E\bigl[V_1^2V_2^2\bigr]
-\E[T]^2,
\\[1ex]
\Cov(S,T)
&=
m\E\bigl[V_1^3\ln(V_1)\bigr]
+m^2\E\bigl[V_1V_2^2\ln(V_1)\bigr]
-\E[S]\E[T],
\\[1ex]
\Var(U^2)
&=
\E[U^4]-\E[U^2]^2,
\\[1ex]
\Cov(S,U^2)
&=
m\E\bigl[V_1U^2\ln(V_1)\bigr]+\E\bigl[U^3\ln(U)\bigr]-\E[S]\E[U^2],
\\[1ex]
\Cov(U,U^2)
&=
\E[U^3]-\E[U]\E[U^2],
\\[1ex]
\Cov(U^2,T)
&=
m\E\bigl[V_1^2U^2\bigr]-\E[U^2]\E[T],
\end{align*}
where
\begin{align*}
\E\bigl[V_1^2(\ln(V_1))^2\bigr]
&=
\frac{1}{2m+1}
\Biggl[
\left(
\psi\!\left(\frac{2m+1}{m}\right)
-\psi\!\left(\frac{3m+1}{m}\right)
\right)^2
\\[-0.2ex]
&\hspace{4.8em}
+\psi_1\!\left(\frac{2m+1}{m}\right)
-\psi_1\!\left(\frac{3m+1}{m}\right)
\Biggr],
\\[1ex]
\E\bigl[V_1V_2\ln(V_1)\ln(V_2)\bigr]
&=
\frac{1}{(m+1)(2m+1)}
\Biggl[
\left(
\psi\!\left(\frac{m+1}{m}\right)
-\psi\!\left(\frac{3m+1}{m}\right)
\right)^2
\\[-0.2ex]
&\hspace{4.8em}
-\psi_1\!\left(\frac{3m+1}{m}\right)
\Biggr],
\\[1ex]
\E\bigl[V_1\ln(V_1)\,U\bigr]
&=
\frac{1}{(m+1)(2m+1)}
\left(
\psi\!\left(\frac{m+1}{m}\right)
-\psi\!\left(\frac{3m+1}{m}\right)
\right),
\\[1ex]
\E\bigl[V_1^4\bigr]
&=
\frac{1}{4m+1},\quad
\E\bigl[V_1^2V_2^2\bigr]
=
\frac{m+1}{(2m+1)(3m+1)(4m+1)},
\\[1ex]
\E\bigl[V_1^3\ln(V_1)\bigr]
&=
\frac{1}{3m+1}
\left(
\psi\!\left(\frac{3m+1}{m}\right)
-\psi\!\left(\frac{4m+1}{m}\right)
\right),
\\[1ex]
\E\bigl[V_1V_2^2\ln(V_1)\bigr]
&=
\frac{1}{(2m+1)(3m+1)}
\left(
\psi\!\left(\frac{m+1}{m}\right)
-\psi\!\left(\frac{4m+1}{m}\right)
\right),
\\[1ex]
\E[U^2]
&=
\frac{1}{2m+1},
\qquad
\E[U^3]
=
\frac{1}{3m+1},
\qquad
\E[U^4]
=
\frac{1}{4m+1},
\\[1ex]
\E\bigl[V_1^2U^2\bigr]
&=
\frac{m+1}{(2m+1)(3m+1)(4m+1)},
\end{align*}
\begin{align*}
\E\bigl[V_1U^2\ln(V_1)\bigr]
&=
\frac{1}{(2m+1)(3m+1)}
\left(
\psi\!\left(\frac{m+1}{m}\right)
-\psi\!\left(\frac{4m+1}{m}\right)
\right),
\\[1ex]
\E\!\left[U^2 \ln(U)\right]&=\frac{1}{2m+1}\left(\psi\!\left(\frac{2m+1}{m}\right)-\psi\!\left(\frac{3m+1}{m}\right)\right),
\\[1ex]
\E\bigl[U^3\ln(U)\bigr]
&=
\frac{1}{3m+1}
\left(
\psi\!\left(\frac{3m+1}{m}\right)
-\psi\!\left(\frac{4m+1}{m}\right)
\right).
\end{align*}
Here, $\psi(x):=\frac{d}{dx}\left(\ln\Gamma(x)\right)$ and $\psi_1(x):=\frac{d}{dx}\psi(x)$ denote the digamma and trigamma functions, respectively.
\end{proposition}

\begin{proof}
Set
\[
S:=\sum_{r=1}^{m+1}V_r\ln(V_r),\qquad
T:=\sum_{r=1}^{m}V_r^2,\qquad
U:=V_{m+1}.
\]
We have
\[
\sum_{r=1}^{m+1}V_r^2=T+U^2
\]
and, since $\sum_{r=1}^{m+1}V_r=1$,
\[
\sum_{r=1}^{m}V_r(1-V_r)=\sum_{r=1}^{m}V_r-\sum_{r=1}^{m}V_r^2=(1-U)-T.
\]
Hence, we obtain
\begin{align*}
b_1
&=
m^2\left(
S-\sum_{r=1}^{m+1}V_r^2+\sum_{r=1}^{m}V_r(1-V_r)+1
\right)
\\
&=
m^2\left(
S-(T+U^2)+(1-U)-T+1
\right)
\\
&=
m^2\bigl(S+2-U-U^2-2T\bigr),
\end{align*}
while
\[
b_2=m(S+1-U).
\]
Therefore,
\begin{align*}
\Var(b_1)
&=
m^4\Var(S-U-U^2-2T)
\\
&=
m^4\Bigl(
\Var(S)+\Var(U)+\Var(U^2)+4\Var(T)
-2\Cov(S,U)
\\[-0.2ex]
&\hspace{4.8em}
-2\Cov(S,U^2)-4\Cov(S,T)
+2\Cov(U,U^2)\\[-0.2ex]
&\hspace{4.8em}
+4\Cov(U,T)
+4\Cov(U^2,T)
\Bigr),
\\[1ex]
\Var(b_2)
&=
m^2\Var(S-U)
=
m^2\Bigl(
\Var(S)+\Var(U)-2\Cov(S,U)
\Bigr),
\\[1ex]
\Cov(b_1,b_2)
&=
m^3\Cov(S-U-U^2-2T,\;S-U)
\\
&=
m^3\Bigl(
\Var(S)+\Var(U)-2\Cov(S,U)-\Cov(S,U^2)
\\[-0.2ex]
&\hspace{4.8em}
+\Cov(U,U^2)-2\Cov(S,T)+2\Cov(T,U)
\Bigr).
\end{align*}
First of all, we have, as in the proof of Theorem \ref{thm:covariance-fixed-point}, for every integer $k\ge1$,
\[
\E[V_r^k]=\frac{1}{km+1}.
\]
In particular, we obtain
\[
\E[T]=m\E[V_1^2]=\frac{m}{2m+1},
\qquad
\E[U]=\E[V_{m+1}]=\frac{1}{m+1},
\]
and
\[
\Var(U)
=
\frac{1/m}{\bigl((m+1)/m\bigr)^2\bigl((2m+1)/m\bigr)}
=
\frac{m^2}{(m+1)^2(2m+1)}.
\]
Next, by the symmetry of the Dirichlet distribution, we have
\[
\E[S]=(m+1)\E[V_1\ln(V_1)].
\]
To compute $\E[V_1\ln(V_1)]$, we use the standard moment formula for the Dirichlet distribution.
If
\[
(X_1,\dots,X_d)\sim \Dir(\alpha_1,\dots,\alpha_d)
\qquad\text{and}\qquad
\alpha:=\alpha_1+\cdots+\alpha_d,
\]
then for $t>-\alpha_1$,
\[
\E[X_1^t]
=
\frac{\Gamma(\alpha)\Gamma(\alpha_1+t)}
{\Gamma(\alpha_1)\Gamma(\alpha+t)}.
\]
Differentiating this identity with respect to $t$ gives
\[
\frac{d}{dt}\E[X_1^t]
=
\E[X_1^t\ln(X_1)].
\]
On the other hand, differentiating the right-hand side yields
\[
\frac{d}{dt}\E[X_1^t]
=
\frac{\Gamma(\alpha)\Gamma(\alpha_1+t)}
{\Gamma(\alpha_1)\Gamma(\alpha+t)}
\bigl(\psi(\alpha_1+t)-\psi(\alpha+t)\bigr),
\]
hence
\[
\E[X_1^t\ln(X_1)]
=
\E[X_1^t]\bigl(\psi(\alpha_1+t)-\psi(\alpha+t)\bigr).
\]
In the case
\[
(V_1,\dots,V_{m+1})\sim \Dir\!\left(\frac1m,\dots,\frac1m\right).
\]
we obtain
\begin{align*}
\E[V_1\ln(V_1)]
&=
\E[V_1]
\left(
\psi\!\left(\frac{1}{m}+1\right)
-
\psi\!\left(\frac{m+1}{m}+1\right)
\right)\\
&=\frac{1}{m+1}
\left(
\psi\!\left(\frac{m+1}{m}\right)
-
\psi\!\left(\frac{2m+1}{m}\right)
\right),
\end{align*}
which yields
\[
\E[S]
=
(m+1)\E[V_1\ln(V_1)]
=
\psi\!\left(\frac{m+1}{m}\right)
-
\psi\!\left(\frac{2m+1}{m}\right).
\]
We next derive the formulas for the second moments of $S$, $T$, and their mixed terms.
Expanding $S^2$, we get
\[
S^2
=
\sum_{r=1}^{m+1}V_r^2(\ln(V_r))^2
+
\sum_{\substack{r,s=1\\ r\neq s}}^{m+1}V_rV_s\ln(V_r)\ln(V_s).
\]
Once again using the symmetry of the Dirichlet distribution, all diagonal terms have the same expectation, and all off-diagonal
terms have the same expectation. Hence, we get
\[
\E\left[S^2\right]
=
(m+1)\E\left[V_1^2(\ln(V_1))^2\right]
+
m(m+1)\E\left[V_1V_2\ln(V_1)\ln(V_2)\right].
\]
Similarly it follows that
\[
SU=\sum_{r=1}^{m+1}V_r\ln(V_r)\,U,
\]
so we have
\[
\E[SU]=
m\,\E\!\left[V_1 \ln(V_1)\,U\right]+\E\!\left[U^2 \ln(U)\right],
\]
and thus
\[
\Cov(S,U)=m\,\E\!\left[V_1 \ln(V_1)\,U\right]+\E\!\left[U^2 \ln(U)\right]-\E[S]\E[U],
\]
For $T$, we have
\[
T^2
=
\sum_{r=1}^{m}V_r^4
+
\sum_{\substack{r,s=1\\ r\neq s}}^{m}V_r^2V_s^2.
\]
Again using the symmetry of the Dirichlet distribution, we obtain
\[
\E\left[T^2\right]
=
m\E\left[V_1^4\right]+m(m-1)\E\left[V_1^2V_2^2\right].
\]
For the mixed term $ST$, we get
\[
ST
=
\sum_{r=1}^{m+1}\sum_{s=1}^{m}V_rV_s^2\ln(V_r).
\]
Among these terms there are $m$ diagonal contributions with $r=s\le m$, each of the form
$V_1^3\ln(V_1)$, and $m^2$ off-diagonal contributions, each having the same expectation as
$V_1V_2^2\ln(V_1)$ by symmetry of the Dirichlet distribution. Hence, it follows
\[
\E[ST]
=
m\E\left[V_1^3\ln(V_1)\right]+m^2\E\left[V_1V_2^2\ln(V_1)\right].
\]
Finally, we have
\[
\Cov(T,U)=m\E\left[V_1^2U\right]-\E[T]\E[U].
\]
Using the Dirichlet moment formula, we get
\[
\E[V_1^2U]
=
\frac{\Gamma\!\left(\frac{m+1}{m}\right)}
{\Gamma\!\left(\frac{m+1}{m}+3\right)}
\frac{\Gamma\!\left(\frac{1}{m}+2\right)}
{\Gamma\!\left(\frac{1}{m}\right)}
\frac{\Gamma\!\left(\frac{1}{m}+1\right)}
{\Gamma\!\left(\frac{1}{m}\right)}.
\]
Using $\Gamma(x+1)=x\Gamma(x)$ and $\Gamma(x+2)=(x+1)x\Gamma(x)$, this simplifies to
\[
\E[V_1^2U]
=
\frac{1}{(2m+1)(3m+1)},
\]
so
\[
\Cov(T,U)
=
\frac{m}{(2m+1)(3m+1)}
-
\frac{m}{(m+1)(2m+1)}.
\]
Next, we turn to the terms involving $U^2$. Since $U=V_{m+1}\sim\Beta(1/m,1)$, we have
\[
\E[U^2]=\frac{1}{2m+1},
\qquad
\E[U^3]=\frac{1}{3m+1},
\qquad
\E[U^4]=\frac{1}{4m+1}.
\]
Using these moments, we can compute
\[
\Var(U^2)=\E[U^4]-\E[U^2]^2
\]
and
\[
\Cov(U,U^2)=\E[U^3]-\E[U]\E[U^2].
\]
Moreover, we have
\[
\Cov(U^2,T)=\E[U^2T]-\E[U^2]\E[T].
\]
Since
\[
T=\sum_{r=1}^{m}V_r^2,
\]
once again, symmetry yields
\[
\E[U^2T]=m\E[V_1^2U^2].
\]
Once again, the Dirichlet moment formula gives us
\[
\E[V_1^2U^2]
=
\frac{\Gamma\!\left(\frac{m+1}{m}\right)}
{\Gamma\!\left(\frac{m+1}{m}+4\right)}
\frac{\Gamma\!\left(\frac{1}{m}+2\right)}
{\Gamma\!\left(\frac{1}{m}\right)}
\frac{\Gamma\!\left(\frac{1}{m}+2\right)}
{\Gamma\!\left(\frac{1}{m}\right)}.
\]
Hence, we get
\[
\E[V_1^2U^2]
=
\frac{m+1}{(2m+1)(3m+1)(4m+1)}
\]
and finally
\[
\Cov(U^2,T)=m\E[V_1^2U^2]-\E[U^2]\E[T].
\]
Finally, we have to consider
\[
\Cov(S,U^2)=\E[SU^2]-\E[S]\E[U^2].
\]
This can be done using
\[
SU^2=\sum_{r=1}^{m+1}V_r\ln(V_r)\,U^2.
\]
Among these terms there are $m$ contributions with $r\le m$, each having the same expectation as
$V_1U^2\ln(V_1)$, and one contribution with $r=m+1$, namely $U^3\ln(U)$. Therefore,
\[
\E[SU^2]
=
m\E[V_1U^2\ln(V_1)]+\E[U^3\ln(U)],
\]
and thus
\[
\Cov(S,U^2)
=
m\E[V_1U^2\ln(V_1)]+\E[U^3\ln(U)]-\E[S]\E[U^2].
\]
It remains to evaluate the logarithmic mixed moments. These are obtained by differentiating
the Dirichlet moment formula. For one coordinate,
\begin{align*}
\lefteqn{\E\left[V_1^a(\ln(V_1))^2\right]}\\
&=
\E\left[V_1^a\right]
\left[
\bigl(\psi(\alpha_1+a)-\psi(\alpha+a)\bigr)^2
+
\psi_1(\alpha_1+a)-\psi_1(\alpha+a)
\right],
\end{align*}
and for two distinct coordinates $i\neq j$,
\begin{align*}
\E\left[V_i^aV_j^b\ln(V_i)\ln(V_j)\right]
=
\E&\left[V_i^aV_j^b\right]
\big[
\bigl(\psi(\alpha_i+a)-\psi(\alpha+a+b)\bigr)\\
&\times\bigl(\psi(\alpha_j+b)
-\psi(\alpha+a+b)\bigr)
-\psi_1(\alpha+a+b)
\big],
\end{align*}
as well as
\[
\E[V_i^aV_j^b\ln(V_i)]
=
\E[V_i^aV_j^b]
\bigl(\psi(\alpha_i+a)-\psi(\alpha+a+b)\bigr).
\]
Substituting $\alpha_i=1/m$ and $\alpha=(m+1)/m$, we obtain
\begin{align*}
\E[V_1^2(\ln(V_1))^2]
&=
\frac{1}{2m+1}
\Biggl[
\left(
\psi\!\left(\frac{2m+1}{m}\right)-\psi\!\left(\frac{3m+1}{m}\right)
\right)^2\\[1ex]
&\hspace{5.8em}+
\psi_1\!\left(\frac{2m+1}{m}\right)-\psi_1\!\left(\frac{3m+1}{m}\right)
\Biggr],
\\[1ex]
\E[V_1V_2\ln(V_1)\ln(V_2)]
&=
\frac{1}{(m+1)(2m+1)}
\Biggl[
\left(
\psi\!\left(\frac{m+1}{m}\right)-\psi\!\left(\frac{3m+1}{m}\right)
\right)^2\\[1ex]
&\hspace{10em}-\psi_1\!\left(\frac{3m+1}{m}\right)
\Biggr],
\\[1ex]
\E[V_1\ln(V_1)\,U]
&=
\frac{1}{(m+1)(2m+1)}
\left(
\psi\!\left(\frac{m+1}{m}\right)-\psi\!\left(\frac{3m+1}{m}\right)
\right),
\\[1ex]
\E[V_1^4]
&=
\frac{1}{4m+1},
\qquad
\E[V_1^2V_2^2]
=
\frac{m+1}{(2m+1)(3m+1)(4m+1)},\\[1ex]
\E[V_1^2U]
&=
\frac{1}{(2m+1)(3m+1)},
\\[1ex]
\E[V_1^3\ln(V_1)]
&=
\frac{1}{3m+1}
\left(
\psi\!\left(\frac{3m+1}{m}\right)-\psi\!\left(\frac{4m+1}{m}\right)
\right),
\\[1ex]
\E[V_1V_2^2\ln(V_1)]
&=
\frac{1}{(2m+1)(3m+1)}
\left(
\psi\!\left(\frac{m+1}{m}\right)-\psi\!\left(\frac{4m+1}{m}\right)
\right),\\[1ex]
\E[U^2]
&=
\frac{1}{2m+1},
\quad
\E[U^3]
=
\frac{1}{3m+1},
\quad
\E[U^4]
=
\frac{1}{4m+1},
\\[1ex]
\E[V_1^2U^2]
&=
\frac{m+1}{(2m+1)(3m+1)(4m+1)},
\\[1ex]
\E[V_1U^2\ln(V_1)]
&=
\frac{1}{(2m+1)(3m+1)}
\left(
\psi\!\left(\frac{m+1}{m}\right)-\psi\!\left(\frac{4m+1}{m}\right)
\right),
\\[1ex]
\E\!\left[U^2 \ln(U)\right]&=\frac{1}{2m+1}\left(\psi\!\left(\frac{2m+1}{m}\right)-\psi\!\left(\frac{3m+1}{m}\right)\right),
\\[1ex]
\E[U^3\ln(U)]
&=
\frac{1}{3m+1}
\left(
\psi\!\left(\frac{3m+1}{m}\right)-\psi\!\left(\frac{4m+1}{m}\right)
\right).
\end{align*}
Substituting these identities into the decompositions for $\Var(S)$, $\Cov(S,U)$,
$\Var(T)$, $\Cov(S,T)$, $\Var(U^2)$, $\Cov(S,U^2)$, $\Cov(U,U^2)$, and
$\Cov(U^2,T)$ proves all stated formulas.
\end{proof}

\end{document}